\documentclass[letterpaper]{amsart}
\usepackage{amssymb,graphicx,color,mathtools,mathrsfs,enumerate} 
\usepackage[all]{xy} 
\usepackage{comment}
\usepackage[colorlinks]{hyperref}
\hypersetup{citecolor=blue}
\DeclareMathOperator{\trace}{Tr}
\DeclareMathOperator{\sub}{sub}

\begin{document}
\begin{abstract}
In this paper, we establish positive results for two spectral inverse problems in the presence of a magnetic potential. Exploiting the principal wave trace invariants, we first observe that on closed Anosov manifolds with simple length spectrum, one can recover an electric and a magnetic (up to a natural gauge) potential from the spectrum of the associated magnetic Schr\"odinger operator. This simple observation extends a particular instance of a recent positive result on the spectral inverse problem for the Bochner Laplacian in negative curvature, obtained by M.Ceki\'c and T.Lefeuvre (2023). Similarly, we prove that the spectrum of the magnetic Dirichlet-to-Neumann map (or magnetic Steklov operator) on a compact Riemannian manifold with boundary determines both a magnetic potential (up to gauge) and an electric potential at the boundary, provided the latter is Anosov with simple length spectrum. Under this assumption, one can actually show that the magnetic Steklov spectrum determines the full Taylor series at the boundary of any smooth magnetic field and electric potential. As a simple consequence, in this case, both an analytic magnetic field and an analytic electric potential are uniquely determined by their Steklov spectrum.
\end{abstract}
\theoremstyle{plain}
\newtheorem{thm}{Theorem}[section]
\newtheorem{prop}[thm]{Proposition}
\newtheorem{lem}[thm]{Lemma}
\newtheorem{clry}[thm]{Corollary}

\newtheorem{hyp}{Assumption}

\theoremstyle{definition}
\newtheorem{rem}[thm]{Remark}
\newtheorem{deft}[thm]{Definition}
\numberwithin{equation}{section}
\newcommand{\eps}{\varepsilon}
\newcommand{\e}{\mathrm{e}}
\renewcommand{\d}{\partial}
\newcommand{\dd}{\mathrm{d}}
\newcommand{\re}{\mathop{\rm Re} }
\newcommand{\im}{\mathop{\rm Im}}
\newcommand{\ran}{\mathop{\rm ran}} 
\newcommand{\I}{\mathrm{I}} 
\newcommand{\R}{\mathbb{R}}
\newcommand{\C}{\mathbb{C}}
\newcommand{\N}{\mathbb{N}} 
\newcommand{\Z}{\mathbb{Z}} 
\newcommand{\D}{\mathcal{C}^{\infty}_0} 
\renewcommand{\O}{\mathcal{O}}
\newcommand{\tr}[1]{\prescript{\rm t}{}{#1}}

\def\restriction#1#2{\mathchoice
              {\setbox1\hbox{${\displaystyle #1}_{\scriptstyle #2}$}
              \restrictionaux{#1}{#2}}
              {\setbox1\hbox{${\textstyle #1}_{\scriptstyle #2}$}
              \restrictionaux{#1}{#2}}
              {\setbox1\hbox{${\scriptstyle #1}_{\scriptscriptstyle #2}$}
              \restrictionaux{#1}{#2}}
              {\setbox1\hbox{${\scriptscriptstyle #1}_{\scriptscriptstyle #2}$}
              \restrictionaux{#1}{#2}}}
\def\restrictionaux#1#2{{#1\,\smash{\vrule height .8\ht1 depth .85\dp1}}_{\,#2}} 
\title[]{Magnetic spectral inverse problems on compact Anosov manifolds}
\author[]{David Dos Santos Ferreira \and Benjamin Florentin}
\address{Université de Lorraine, CNRS, Inria, IECL, F-54000 Nancy, France}
\email{ddsf@math.cnrs.fr}
\address{Université de Lorraine, CNRS, Inria, IECL, F-54000 Nancy, France}
\email{benjamin.florentin@univ-lorraine.fr}

\maketitle
\setcounter{tocdepth}{1} 
\tableofcontents

\section{Introduction}

\subsection{Spectral inverse problems on compact manifolds}\label{Section 1.1}

In a famous article in Topology \cite{GK1}, Guillemin and Kazhdan addressed the spectral inverse problem of determining a metric $g$ or a potential $q$ from the spectrum of the Schr\"odinger operator $-\Delta_g+q$ on a closed Riemannian manifold. 
They were able to show the absence of nontrivial isospectral deformations on surfaces of negative curvature, and under the additional assumption that the length spectrum is simple, the identifiability of the potential. A metric is said to have \textit{simple length spectrum} if all its closed geodesics (including iterated ones) have different lengths. It is well-known that this assumption is satisfied by generic Riemannian metrics (cf. \cite{Ab}, \cite[Lemma 4.4.3]{Kli} or \cite{An}).
The restriction of spectral inverse problems to isospectral deformations is, in fact, a linearization of the problem of determining the principal symbol of the Schr\"odinger operator.  
Identifying potential from the spectrum also leads to a linear tomography problem involving the geodesic X-ray transform.
The article \cite{GK2}, which followed \cite{GK1}, dealt with the higher dimensional case.  
Many similar results, in particular compactness results of isospectral sets (see \cite{BrPeYa} and \cite{Sha}), use spectral invariants coming from the heat trace. 
Unfortunately these invariants only provide information about averages (over the whole manifold) of local invariants, which are geometric quantities given by polynomials in the underlying metric (cf. the instructive surveys \cite{Zeld,Zeld15} or \cite{DatHez}).
Both results in \cite{GK1,GK2} as well as in \cite{G1,G2,G3} are rather based on the wave trace formula of Duistermaat and Guillemin \cite{DG} which provides a careful analysis of the singularities of the trace of the wave operator as a distribution in time.
This trace formula gives information on the average of the subprincipal symbol over periodic Hamiltonian integral curves of the principal symbol,  i.e.  cogeodesic curves, provided these curves are isolated and satisfy some nondegeneracy assumptions (see Section~\ref{SecDGTrace}).

More recently, Paternain,  Salo and Uhlmann \cite{PatSalUhl2} were able to extend the result on absence of isospectral deformations of \cite{GK1} to the case of closed Anosov surfaces. This relies on a difficult injectivity result for the geodesic $X$-ray transform on tensors of order 2 on closed surfaces with Anosov geodesic flow. Recall that a flow $\left(\phi_{t}\right)$ generated by a smooth vector field $H$ on a closed manifold $\mathcal{M}$ is Anosov if there is a continuous flow-invariant splitting 
\[ T\mathcal{M}=\mathbb{R}H\oplus E_{s}\oplus E_{u} \]
of the tangent bundle $T\mathcal{M}$, into three flow-invariant subbundles, respectively, tangential to the flow direction, exponentially contracting and expanding. If $\mathcal{M}$ denotes the unit (co)tangent bundle of a Riemannian manifold and $H$ the associated geodesic vector field then the manifold is said to be Anosov. It is well-known that closed manifolds with negative sectional curvature are Anosov (see \cite[Theorem 5.2.4]{FisherHass}, \cite[Section 17.6]{KH} or \cite{Kni} for a proof).
In all the aforementioned results using wave trace formula techniques, the subprincipal symbol of the operator under scope vanishes.
This cancellation both simplifies the spectral inverse problem but also provides less information on the quantities to recover from the spectrum.
Using this observation, the second author initiated in \cite{Flo} the analysis of the Steklov spectral inverse problem from the Duistermaat and Guillemin's trace formula. 
This is indeed an instance where the subprincipal symbol does not vanish and leads to a nonlinear tomography inverse problem, in contrast with the classical spectral inverse problem on the Laplacian.

In this article, we are interested in spectral inverse problems for other operators whose subprincipal symbol is nonzero but are parametrized by a smooth $1$-form, while the metric is fixed; the first one is the magnetic Schr\"odinger operator (see Section \ref{Section 1.3}) and the second is the magnetic Steklov operator (see Section \ref{Section 1.4}).
One of the epitome of spectral inverse results on magnetic operators is 
\cite{Shi} where it is shown that if the first eigenvalue of a magnetic Schr\"odinger operator is zero, that is, is equal to the first eigenvalue of a Schr\"odinger operator with trivial magnetic field, then the magnetic potential is zero modulo the natural gauge invariance.  
This can be seen as a spectral inverse problem with spectral datum the lowest eigenvalue. 
Recently, Ceki\'c and Lefeuvre \cite{CekLef} obtained that the spectrum of the Bochner Laplacian on closed manifolds with negative curvature and simple length spectrum determines a connection, up to gauge, under a low rank assumption. 
This is closely related to the framework of the magnetic Laplacian (cf. Section \ref{Sec:Connexion}). Let us also mention the earlier results of \cite{G4}, \cite{EskRals1} on the $2$-torus. We are not aware of other positive results on spectral inverse problems for the magnetic Schr\"odinger operator in the case of closed Riemannian manifolds. 
Spectral inverse problems in the presence of a magnetic field have been studied quite extensively but rather in the case of manifolds with boundary,  where the spectral data include information about the (Neumann) traces of the (Dirichlet) eigenfunctions in addition to the sequence of eigenvalues. 
This type of spectral inverse problems are usually referred to as Borg-Levinson theorems \cite{BCDSFKS, Se} and are also related to boundary inverse problems \cite{NSU,DSFKSU} (cf. survey \cite{SalSurvey} on the subject). 

One of our results, stated in \textbf{Theorem \ref{Th 1.1}}, concerns the identifiability, up to a natural gauge, of the magnetic and electric potentials from the spectrum of their magnetic Schr\"odinger operator on a closed Anosov manifold, provided the length spectrum is simple. 
The latter assumption is known to be necessary, as there are counterexamples of isospectral electric potentials when the length spectrum is not simple (cf. \cite{Kuw}, \cite{Br}).
In particular, our result extends a special instance of that obtained by Ceki{\'c} and Lefeuvre on the magnetic Laplacian in the aforementioned \cite{CekLef}.

Since the initial motivation was the study of the Steklov spectrum from the point of view of the second author \cite{Flo}, our most important results concern the magnetic Steklov inverse problem and give a similar boundary determination statement, namely the recovery of the jet at the boundary of an electric potential and a magnetic field from the spectrum of its magnetic Dirichlet-to-Neumann map (see Section \ref{Section 5}), provided that the boundary is Anosov with simple length spectrum. As a simple corollary, an analytic magnetic field is entirely determined by its Steklov spectrum.
To our knowledge, these are the first positive results for the magnetic Steklov inverse problem in a general framework. Unlike the case of the magnetic Laplacian mentioned above, the literature on the Steklov spectrum in the presence of a magnetic field seems rather limited at the moment ; let us mention the very recent \cite{CGHP} --- which is the analogue of Shigekawa's result \cite{Shi} for the magnetic Steklov lowest eigenvalue ---, \cite{ProSav} on the direct problem, \cite{LiuTan}, \cite{HelfNic} on the Steklov heat trace asymptotics, and \cite{CekSif} providing important spectral invariants in dimension $2$. An important observation from our work is that, in contrast to the magnetic Schr\"odinger operator which is just a differential operator of order $2$, the magnetic DN map is a classical pseudodifferential operator of order $1$ whose full symbol, considered in suitable local coordinates, contains much more information about the potentials. More precisely, in boundary normal coordinates, each of its homogeneous terms of order $-j\leq -1$ involves the normal derivative at the boundary of both magnetic and electric potentials, respectively of order $j$ and $j-1$ (see \textbf{Lemma \ref{Lemma 5.6}}).
Consequently, using the approach developed in \cite[Section IV]{Flo}, we are able to recover, modulo gauge, the full jet at the boundary of these potentials by induction on the order $j$. However, in this paper, the procedure is much more delicate because of the gauge on the magnetic potential that must be managed at each step (cf. Section \ref{Section 5.3}). 

Note that the only source of non-uniqueness in these magnetic spectral inverse problems arises from the natural gauge invariance on magnetic potentials. In particular, on manifolds with non-trivial topology, gauge-equivalent magnetic potentials may produce the same magnetic field while remaining spectrally indistinguishable. This phenomenon is closely related to the \emph{Aharonov–Bohm effect} (see \cite{AhBo}, \cite{Esk}), which represents the fact that a zero magnetic field does not necessarily imply the invisibility of the magnetic potential. Indeed, the spectrum only detects the magnetic fluxes modulo $2\pi$ (cf. \textbf{Corollary \ref{Cor 2.2}} and \textbf{Lemma \ref{Lemma 4.2}}). On this subject, the reader may also consult \cite[Proposition 1.1]{Helf} and the discussion therein.

\subsection{Notations}\label{Sect 1.2}

Throughout the article, we will denote closed Riemannian manifolds by $(M,g)$ and compact Riemannian manifolds with boundary by $(N,h)$. For the latter, we will denote $h^{\circ}$ the induced metric on the boundary. 
\\In addition, we will adopt the standard notations $$Hu:=\mathcal{L}_{H}u=\restriction{\frac{d}{dt}}{t=0}\phi_{t}^{*}u$$ for the Lie derivative of the function $u$ along the vector field $H$ generated by $\phi_{t}$. Unless otherwise stated, $\phi_{t}$ will always be the (co)geodesic flow on the unit cotangent bundle $S^{*}M$ and $H$ its generator.
\\Finally, denote by $\Psi^{k}(M)$ (resp. $\Psi_{cl}^{k}(M)$) the set of pseudodifferential (resp. classical pseudodifferential) operators of order $k$ on the closed manifold $M$. Recall that a pseudodifferential operator $P\in\Psi^{m}(M)$ is said to be \textit{classical} if its full symbol $\sigma_{P}^{\rm full}$ considered in any coordinate chart is \textit{polyhomogeneous} i.e admits an asymptotic expansion $$\sigma_{P}^{\rm full}\left(x,\xi \right)\sim \sum_{j=0}^{+\infty}p_{m-j}(x,\xi)\hspace{0.1cm},$$ \medskip\noindent with $p_{m-j}$ homogeneous of degree $m-j$ in $\xi$ (see \cite[Sect.18.1]{Horm1}).
\medskip\noindent\\We will denote by $\sigma_{P}$ and $\sub(P)$ repectively the principal and subprincipal symbol of $P$, given in any local coordinates by $p_{m}$ and

$$p_{m-1}-\frac{1}{2i}\sum_{j=1}^{n}\frac{\partial^{2}p_{m}}{\partial x_{j}\partial\xi_{j}}\hspace{0.1cm}\cdot$$ 

\medskip\noindent It is well-known that these quantities are invariant to coordinate changes, we refer to \cite{DH} and \cite{Horm1,Horm2} for details. Since the symbols of linear operators on $M$ live in the cotangent bundle $T^{*}M$, our natural preference in this paper is to formulate the results and perform calculations in the cotangent space rather than the tangent space.

Throughout this article, unless otherwise specified, $a$, $\tilde{a}$ and $q$, $\tilde{q}$ will refer respectively to smooth magnetic potentials and electric potentials either in $N$ or in $M$, depending on the problem under consideration (see the two subsections \ref{Section 1.3} and \ref{Section 1.4} below). Finally, the following notations

$$\begin{array}{l|rcl}
\pi_{0}^{*}: & C^{\infty}(M) & \longrightarrow & C^{\infty}(S^{*}M) \\
			& q & \longmapsto & (x,\xi)\mapsto q(x) \end{array}$$
and            
            $$\begin{array}{l|rcl}
\pi_{1}^{*}: & C^{\infty}(M,S^{*}M) & \longrightarrow & C^{\infty}(S^{*}M) \\
			& a & \longmapsto & (x,\xi)\mapsto a(x)\xi^{\sharp} \end{array}\hspace{0.1cm},$$

\medskip\noindent may be used, particularly in Section \ref{Section 4} and \ref{Section 5}, to denote the natural lifts.

\subsection{The magnetic Schr\"odinger operator}\label{Section 1.3}
Let $(M,g)$ be a smooth clo-\\sed Riemannian manifold of dimension $n \geq 2$.  \\A magnetic potential is a real-valued one form $a \in \Omega^1(M;\R)$
and the corresponding magnetic field is the exterior differential $b = \dd a \in \Omega^2(M;\R)$.
If $q \in C^{\infty}(M;\R)$ is a smooth electric potential, we consider the magnetic Schr\"odinger operator 
\[ P_{a,q} = (\dd + i a )^*(\dd + i a) + q\hspace{0.1cm},\] 
where the adjonction is taken with respect to the scalar product on forms given by
\[ (\chi|\omega) = \int_M \langle \chi, \bar{\omega}\rangle \, \dd \mathrm{vol}_g = \int_M \chi \wedge *\bar{\omega}\hspace{0.1cm}.  \] 
In local coordinates, with the classical notation $D_j := \frac{1}{i} \frac{\d}{\d x_j}$ and
denoting $a=\sum_{i=1}^{n}a_{j}dx_{j}$,
the magnetic Schr\"odinger operator reads 
\begin{align*}
      P_{a,q} = \frac{1}{\sqrt{\det g}} \sum_{j,k=1}^n \big(D_j+a_j\big)\Big( \sqrt{\det g} \, g^{jk} \big(D_k+a_k\big)\Big) + q \hspace{0.1cm}.
\end{align*}
If we define the domain of $P:=P_{a,q}$ to be 
$$D(P) = \big\{u \in H^1(M) : Pu \in L^2(M)\big\}\hspace{0.1cm},$$
then $(P,D(P))$ is a self-adjoint elliptic operator, thus it has a compact resolvent and, therefore, only a pure point spectrum.
Besides, since
\[ (Pu|u) + c\|u\|^2 = \|\dd u + u a\|^2 + ((q+c) u|u)  \geq 0 \]
with $c=\min q$,  the eigenvalues are contained in the interval $[-c,\infty)$.
Assuming that $q$ is non-negative or, if necessary, translating, the spectrum consists in a sequence of eigenvalues
\[ 0\leq \lambda_{0}(P)\leq\lambda_{1}(P)\leq\lambda_{2}(P)\leq...\rightarrow \infty\hspace{0.1cm}. \]

\subsection{The magnetic DN map}\label{Section 1.4}

Let $(N,h)$ be a smooth compact Riemannian manifold of dimension $n\geq 2$ with boundary and $a\in\Omega^1(N;\R)$, $q\in C^{\infty}(N;\R)$ respectively (smooth) magnetic and electric potentials on $N$. 
The magnetic Dirichlet-to-Neumann map (DN map) $\Lambda_{a,q}$ associated with $a$ and $q$ is the DN map associated with the magnetic Schr\"odinger operator :
$$\begin{array}{l|rcl}
\Lambda_{a,q} : & H^{1/2}(\partial N) & \longrightarrow & H^{-1/2}(\partial N) \\
			& f & \longmapsto & \restriction{\left(\dd u+iau\right)}{\partial N}\left(\nu\right) =\restriction{\left(\partial_{\nu}u+i\langle a,\nu\rangle u\right)}{\partial N} \end{array}\hspace{0.1cm},$$ 
where $u$ is the unique solution in $H^{1}(N)$ of the Dirichlet problem
$$\left\lbrace
\begin{aligned}
P_{a,q}u &= 0 \\
\restriction{u}{\partial N} &= f \\
\end{aligned}\right.\hspace{0.1cm},$$
provided that $0$ is not a Dirichlet eigenvalue of $P_{a,q}\hspace{0.1cm}.$ The magnetic DN map is a classical, elliptic and formally self-adjoint pseudodifferential operator of order $1$ on the closed manifold $\partial N$ (cf. \cite{DSFKSU} or \cite{Cek}). \\Since the unbounded operator $\left(\Lambda_{a,q},H^{1}(\partial N)\right)$ is self-adjoint elliptic, the spectrum of $\Lambda:=\Lambda_{a,q}$ is discrete and is given by a sequence of eigenvalues 
$$\sigma_{1}(\Lambda)\leq\sigma_{2}(\Lambda)\leq...\rightarrow \infty\hspace{0.1cm}.$$
This sequence is called the magnetic Steklov spectrum, as it is equivalently the sequence of eigenvalues of the \textit{magnetic Steklov eigenvalue problem} : 
$$\left\lbrace
\begin{aligned}
P_{a,q}u &= 0\hspace{0.2cm}\text{in $N$}\\
\Lambda_{a,q}u &= \sigma u\hspace{0.2cm}\text{on $\partial N$} \\
\end{aligned}\right.\hspace{0.1cm}\cdot$$

\subsection{Main results}

\medskip\noindent Because of the existence of classical counterexamples \cite{V,Sun,Br} on the determination of the metric and the electric potentials from the spectrum of the associated Schr\"odinger operator, we need additional geometric conditions. More precisely, those are assumptions on the closed orbits of the geodesic flow required to exploit the wave trace formula (cf. Section \ref{SecDGTrace}).
Throughout the paper, we will consider a suitable framework where these assumptions are verified. As already mentioned, we will assume that the cogeodesic flow is Anosov with simple length spectrum.  
Since there is a natural closed gauge invariance in  magnetic spectral inverse problems (see Subsection \ref{SubSecGauge}), one can only hope to recover the magnetic potential modulo gauge, and then the magnetic field. 

The main results (Theorems \ref{Th 1.2} and \ref{Th 1.3}) concern boundary determination of the electric potential and the magnetic field from the magnetic Steklov spectrum. Provided that suitable gauge transformations have been fixed for each magnetic potentials (we will refer to these potentials as \emph{normalized}, see \textbf{Remark \ref{Rk 5.1}} and the discussion above it), we can first show :

\begin{thm}\label{Th 1.2}
     Let $(N,h)$ be a compact Riemannian manifold of dimension $n \geq 3$ with boundary $M:=\partial N$ and $(a,q)$, $(\tilde{a},\tilde{q})$ be two pairs of potentials in $N$ such that $a$ and $\tilde{a}$ are normalized. Assume that the cogeodesic flow $\phi_t$ of $M$ is Anosov with simple length spectrum.  \\If the spectra of the magnetic DN maps $\Lambda_{a,q}$ and $\Lambda_{\tilde{a},\tilde{q}}$ coincide 
     \[ \big(\sigma_k(\Lambda)\big)_{k \in \N} = \big(\sigma_k(\tilde{\Lambda})\big)_{k \in \N} \hspace{0.1cm},\] 
     
     \medskip\noindent then the electric potentials at the boundary are equal,
     $$\restriction{\tilde{q}}{M}=\restriction{q}{M}$$ and the magnetic potentials at the boundary are gauge-equivalent, namely there exists a smooth function $\vartheta \in C^{\infty}(M,\mathbb{S}^1)$ with values in the unit circle such that 
     \[ \restriction{\tilde{a}}{M} = \restriction{a}{M}-i\bar{\vartheta}\dd \vartheta\hspace{0.1cm}. \]
     
     \medskip\noindent In particular, the magnetic fields coincide at the boundary
     \[  \restriction{\dd \tilde{a}}{M} =  \restriction{\dd a}{M} \hspace{0.1cm}\cdot\]
\end{thm}

\medskip\noindent As mentioned at the end of the introduction, it is actually possible to recover the full jet at the boundary of the magnetic field and of the electric potential from the Steklov spectrum, thereby obtaining identifiability in the real-analytic setting, this is the object of the next result. In the following, we will use the notation $\tilde{f}\simeq f$ to denote that $\tilde{f}$ and $f$ have the same Taylor series at a point $p$ in some local coordinates. Note that in \textbf{Theorem \ref{Th 1.3}} below, we do not need to assume that magnetic potentials are normalized contrary to \textbf{Theorem \ref{Th 1.2}} (cf. \textbf{Remark \ref{Rk 5.1}}).

\begin{thm}\label{Th 1.3}
 Let $(N,h)$ be a compact Riemannian manifold of dimension $n \geq 3$ with boundary $M:=\partial N$  and $(a,q)$, $(\tilde{a},\tilde{q})$ be two pairs of potentials in $N$. Assume that the cogeodesic flow $\phi_t$ of $M$ is Anosov with simple length spectrum.  \\If the spectra of the magnetic Steklov operators $\Lambda_{a,q}$ and $\Lambda_{\tilde{a},\tilde{q}}$ coincide 
 \[ \big(\sigma_k(\Lambda)\big)_{k \in \N} = \big(\sigma_k(\tilde{\Lambda})\big)_{k \in \N}\hspace{0.1cm}, \]   

\medskip\noindent then in boundary normal coordinates at any point $p\in M$,
$$\dd\tilde{a}\simeq \dd a\hspace{0.3cm}\text{and}\hspace{0.3cm}\tilde{q}\simeq q\hspace{0.1cm}.$$

\medskip\noindent In particular, if $a$, $\tilde{a}$ and $q$, $\tilde{q}$ are real-analytic and $N$ is connected, then $$\dd\tilde{a}=\dd a\hspace{0.3cm}\text{and}\hspace{0.3cm}\tilde{q}=q\hspace{0.1cm}.$$

\end{thm}

\medskip\noindent This result is established in Section \ref{Section 5.3} by \textbf{Theorem \ref{Th 5.9}} and generalizes \cite[Theorem I.4]{Flo}, which concerns the determination of the jet at the boundary of an electric potential from the Steklov spectrum. Regarding the magnetic potentials, let us point out that we are unable to recover a global gauge equivalence on the whole manifold $N$. This has to do with the difficulty of extending an analytic function to the whole manifold because of the non-trivial topology. 
A key step in the proof of all these main results is the conversion of the spectral inverse problem into a tomography inverse problem. In the case of \textbf{Theorem \ref{Th 1.2}} and \textbf{Theorem \ref{Th 1.1}}, the corresponding tomography inverse problem can be solved using a standard argument of Paternain \cite{Pat} in the context of transparent connections, already used in \cite{Flo}. In this work, however, we adopt a different approach and present an alternative proof based on a suitable modification of a classical $L^{2}$-energy identity. We refer the reader to Section \ref{Section 3} for more details. 

Using the same approach, one can easily prove the following theorem --- which extends to Anosov manifolds a particular instance of \cite[Theorem 1.1]{CekLef} on recovering a magnetic potential from knowledge of the spectrum of the magnetic Laplacian. 

\begin{thm}\label{Th 1.1}
     Let $(M,g)$ be a closed Riemannian manifold of dimension $n \geq 2$ and $(a,q)$, $(\tilde{a},\tilde{q})$ be two pairs of potentials in $M$. Assume that the cogeodesic flow $\phi_t$ is Anosov with simple length spectrum.  \\If the spectra of the magnetic Schr\"odinger operators $P_{a,q}$ and $P_{\tilde{a},\tilde{q}}$ coincide 
     \[ \big(\lambda_k(P)\big)_{k \in \N} = \big(\lambda_k(\tilde{P})\big)_{k \in \N} \hspace{0.1cm},\] 
     
     \medskip\noindent then the electric potentials are equal,
     $$\tilde{q}=q$$ and the magnetic potentials are gauge-equivalent, namely there is a smooth function $\vartheta \in C^{\infty}(M,\mathbb{S}^1)$ with values in the unit circle such that 
     \[ \tilde{a} = a -i\bar{\vartheta} \, \dd \vartheta \hspace{0.1cm}.\]
     In particular, the magnetic fields coincide  
     \[  \dd \tilde{a} =  \dd a \hspace{0.1cm}.\]    
\end{thm}
\medskip\noindent Note that this result also extends \cite[Theorem 3(a)]{Gui} to Schr\"odinger operators with non-zero magnetic potential.

\section{Wave trace formulas}
\label{SecDGTrace}

This section introduces the spectral invariants on which this paper is based, provided by the Duistermaat-Guillemin trace formula \cite[Theorem 4.5]{DG}. The latter is related to the trace of the wave group
$$U_{0}(t):=\e^{itP}$$
associated with a self-adjoint elliptic pseudodifferential operator $P\in\Psi_{1}(M)$ of order $1$ on the closed Riemannian manifold $(M,g)$. In this section, we will call \textit{bicharacteristic} of $P$ any integral curve of the Hamiltonian vector field associated with the principal symbol of $P$. Throughout the paper, bicharacteristics will be closed orbits of the geodesic flow of $M$ since $P$ will denote either the square root of a magnetic Schr\"odinger operator $P_{a,q}$ on $M$ or a magnetic DN map $\Lambda_{a,q}$ on a manifold $(N,h)$ with boundary $M:=\partial N$.

\begin{thm}[\cite{DG}]\label{Th 2.1}
    Suppose that $P$ only admits a finite number of periodic bicharacteristics $\gamma_{1},...,\gamma_{r}$ of period $T$ and that each of these curves is non-degenerate. Then
   $$\trace \left(U_{0}(t)\right)\underset{t\rightarrow T}{\sim} \hspace{0.1cm}\frac{1}{t-T}\sum_{j=1}^{r}\frac{\lvert T_{j}^{\sharp}\rvert \e^{im_{j}\frac{\pi}{2}} \e^{iT\hspace{0.05cm}\overline{\gamma_{j}}}}{\lvert \det\left(Id-\mathcal{P}_{j}\right)\rvert^{1/2}}\hspace{0.1cm},$$ where $$\overline{\gamma_{j}}:=\frac{1}{T}\int_{\gamma_{j}}\sub(P)\hspace{0.1cm},$$ and $\mathcal{P}_{j}$, $T_{j}^{\sharp}$, $m_{j}$ are respectively the linearized Poincaré map, the primitive period and the Maslov index of $\gamma_{j}$.
\end{thm}

\medskip\noindent We refer to the original article \cite{DG} and to \cite[Section II.2]{Flo} for details and definitions in the Riemannian case. The important fact is that it leads to exploitable spectral invariants under suitable dynamical assumption on the geodesic flow :  

\begin{clry}\label{Cor 2.2}
In the class of metrics $g$ such that $g$ is Anosov with simple length spectrum on $M$, the spectrum of $P$ determines for any period $T$ of periodic geodesics $\gamma$ of $M$, the quantity 
$$c_{T,-1}=\frac{\lvert T^{\sharp}\rvert \exp\left(i\hspace{0.05cm}\int_{\gamma} \sub(P)\right)}{\lvert \det\left(Id-\mathcal{P}_{\gamma}\right)\rvert^{1/2}}\hspace{0.1cm}\cdot$$
In particular, if $P$ and $\widetilde{P}$ are either magnetic Schr\"odinger operators on $M$ or magnetic DN maps on $(N,h)$ with boundary $M:=\partial N$, and have the same spectrum, then under the above assumptions on the metric $g$, 
$$\int_{\gamma} \left(\sub(P)-\sub(\widetilde{P})\right)\in 2\pi\Z\hspace{0.1cm},$$ 
for any closed orbit $\gamma$ of the geodesic flow of $M$.
\end{clry} 
 
\medskip\noindent This result will be used to prove \textbf{Theorem \ref{Th 1.1}} and \textbf{Theorem \ref{Th 1.2}} respectively in Section \ref{Section 4} and Section \ref{Section 5}. Next, in order to recover the full jet at the boundary of both the magnetic field and the electric potential in \textbf{Theorem \ref{Th 1.3}}, we will use in the same way as in \cite[Section IV]{Flo}, the adapted version of the Duistermaat-Guillemin trace formula formulated in \cite[Theorem IV.2]{Flo}. Considering pseudodifferential operators $P\in\Psi^{1}(M)$ and $Q\in\Psi^{-k}(M)$, $k\in\mathbb{N}^{*}$, such that $P$ and $P+Q$ are self-adjoint operators, this formula is obtained by comparing the asymptotic expansions for the trace of the unitary wave groups $U_{0}(t)$ and $$U(t):=\e^{it\left(P+Q\right)}\hspace{0.1cm},$$ 
giving an explicit expression for the leading term :

\begin{thm}[\cite{Flo}]\label{Th 2.3}
Let $k\in\mathbb{N}^{*}$, $P\in\Psi_{cl}^{1}(M)$ and $Q\in\Psi_{cl}^{-k}(M)$ such that $P$ and $P+Q$ are self-adjoint operators. Assume $P$ is elliptic.
If $P$ only admits a finite number of closed bicharacteristic curves $\gamma_{1},...,\gamma_{r}$ of period $T$ and each of these curves are non-degenerate, then 
   $$ \trace \left(U(t)-U_{0}(t)\right)\underset{t\rightarrow T}{\sim} \hspace{0.1cm}\frac{1}{t-T}\sum_{j=1}^{r}\frac{\lvert T_{j}^{\sharp}\rvert \e^{im_{j}\frac{\pi}{2}}\e^{iT\hspace{0.05cm}\overline{\gamma_{j}}}}{\lvert \det\left(Id-\mathcal{P}_{j}\right)\rvert^{1/2}}\hspace{0.1cm}i\int_{\gamma_{j}}\sigma_{Q}\hspace{0.1cm},$$ 
 where $$\overline{\gamma_{j}}:=\frac{1}{T}\int_{\gamma_{j}}\sub(P)\hspace{0.1cm},$$ and $\mathcal{P}_{j}$, $T_{j}^{\sharp}$, $m_{j}$ are respectively the linearized Poincaré map, the primitive period and the Maslov index of $\gamma_{j}$.
\end{thm}

\medskip\noindent As explained in \cite[Section IV]{Flo}, this theorem is a generalization of \cite[Theorem 4]{Gui}.

\begin{clry}\label{Cor 2.4}
If $P$ is either the square root of a magnetic Schr\"odinger operator on $(M,g)$ or a magnetic DN map on $(N,h)$ with boundary $M:=\partial N$, and has the same spectrum as $P+Q$, then under the assumption that $g$ is Anosov with simple length spectrum,
$$\int_{\gamma}\sigma_{Q}=0$$

\medskip\noindent for any closed orbit $\gamma$ of the geodesic flow of $M$.
\end{clry}

\medskip\noindent This result, which is interesting in itself, will be applied in Section \ref{Section 5.3} to linearize the problem of recovering the full jet at the boundary of $a$ and $q$ from the magnetic Steklov spectrum, by reducing it to a linear tomography inverse problem.
 
\section{A transport equation with real-valued  potential}\label{Section 3}

This section deals with the transport-type equation that naturally appears when we want to specify the information on the X-ray transform given by the principal wave trace invariants in \textbf{Corollary \ref{Cor 2.2}}. Indeed, using the Liv\v sic theorem for smooth cocycles (see \cite[Theorem 2.2]{Pat} and references therein), one can prove the following : 

\begin{prop}\label{Th 3.1}
      Let $(M,g)$ be a closed Anosov manifold and $f:S^{*}M\rightarrow\mathbb{R}$ a smooth function. If for any closed orbit $\gamma$ of the cogeodesic flow,
     \[ \int_{\gamma} f \in 2\pi\Z \hspace{0.1cm},\]  
 then there exists a smooth  function $u : S^{*}M \to \mathbb{S}^1$ such that 
\[ Hu + i f u = 0 \hspace{0.1cm}.\label{eq1}\] 
\end{prop}

\medskip\noindent We are therefore reduced to studying a transport-type equation with real-valued potential. The following result extends the one formulated in \cite[Theorem III.4]{Flo} to closed manifolds without conjugate points and to potentials of degree $1$ (in the sense of decomposition in spherical harmonics) :

\begin{thm}
\label{RigidityTransportEq}
     Let $(M,g)$ be a closed Riemannian manifold without conjugate points. 
     Let $f \in C^{\infty}(S^*M;\R)$ be a smooth real-valued function,  if there exists a non-trivial solution  $u \in H^2(S^*M;\mathbb{S}^{1})$ of the transport equation with potential $f$
     \[ Hu+ifu = 0 \]
     and if $f=f_0+f_1$ is the sum of a function $f_0 \in C^{\infty}(M)$ and a one-form $f_1 \in \Omega^1(M)$ then $f_0=0$,  $f_1$ is closed and $u$ is a function on $M$. 
\end{thm}
\medskip\noindent In the case where the geodesic flow is Anosov, it is easy to prove this result with the same argument mentioned in \cite[Theorem III.4]{Flo}, adapted from \cite[Theorem 3.2]{Pat}. This argument actually still works to prove \textbf{Theorem \ref{RigidityTransportEq}} but only if the manifold is not the $2$-torus.
\\In the following, we present a new proof, rather in the spirit of the famous Pestov identity in relation to transport equations (see, for example, \cite{PatSalUhlbook}), which allows us to formulate the result more naturally on manifolds without conjugate points. Recall that Anosov manifolds do not have conjugate points (cf. \cite{Kl}).
We believe this alternative proof is instructive and hope that it may be useful in different contexts, involving gauge invariances. In order to describe the idea, we begin with the case of dimension $2$, since the geometry of the unit cotangent bundle is much nicer there.

\subsection{The two-dimensional case}

In dimension two, one can choose local isothermal coordinates in which the metric reads
\[ g = \e^{2\varphi} \big(\dd x_1^2 + \dd x_2^2\big) \]
and the symbol $p$ of the Laplacian 
\[ p(x,\xi) = |\xi|_{x}^2 =  \e^{-2\varphi} \big(\xi_1^2+\xi_2^2\big)\hspace{0.1cm}. \] 
We denote by
\[ K = -\Delta_g \varphi\]

\medskip\noindent the Gauss curvature%
\footnote{That is one half of the scalar curvature.} 
of the manifold $(M,g)$. 
The Hamiltonian vector field of the symbol $p$ takes the simple form
\begin{align*}
     H_p &= \frac{\d p}{\d \xi} \cdot \frac{\d}{\d x} - \frac{\d p}{\d x} \cdot \frac{\d}{\d \xi} \\
     &= 2p(x,\xi)\bigg(\frac{\xi_1}{\xi_1^2+\xi_2^2}  \frac{\d}{\d x_1} + \frac{\xi_2}{\xi_1^2+\xi_2^2} \frac{\d}{\d x_2} +  \d_{x_1}\varphi  \frac{\d}{\d \xi_1} + \d_{x_2}\varphi \frac{\d}{\d \xi_2}\bigg)\hspace{0.1cm}. 
\end{align*}
The group $(\R^*_+,\times)$ acts on the cotangent bundle by multiplication in the fibre variables
\[ M_{\lambda}(x,\xi) = (x,\lambda \xi) \]
and the Euler (or radial, or canonical) vector field is the infinitesimal generator of the one parameter group $(M_{\e^t}^*)_{t \in \R}$ 
\[ \varrho u = \frac{\d}{\d t} M_{\e^t}^*u \big|_{t =0}\hspace{0.1cm}. \] 
It reads in local coordinates
\[ \varrho = \xi_1 \frac{\d}{\d \xi_1} + \xi_2 \frac{\d}{\d \xi_2}\hspace{0.1cm}.  \] 
Similarly rotations act on the cosphere bundle 
\[ S^*M = \big\{(x,\xi) \in T^*M :  \xi_1^2+\xi_2^2 = \e^{2\varphi(x)}\big\} \] 
in the fiber variables
\begin{align*}
      (x,\xi) \mapsto (x,R_{\theta}\xi), \quad R_{\theta} = \begin{pmatrix} \cos \theta & -\sin \theta \\ \sin \theta & \cos \theta \end{pmatrix}
\end{align*}
and we denote its infinitesimal generator by $V$ 
\[ Vu :=\frac{\d}{\d \theta} u(x,R_{\theta}\xi)|_{\theta=0}\hspace{0.1cm}.   \] 
The unit cotangent bundle can be parametrized by the angle variable $\theta$
\begin{align*}
      S^*M \ni  \xi = (\xi_1,\xi_2) = \e^{\varphi}(\cos \theta,\sin \theta)
\end{align*}
and the Euler vector field $\varrho$ is normal to the tangent space of the cosphere bundle 
\[ T_{(x,\theta)}(T^*M) = \R \varrho \oplus T_{(x,\theta)}(S^*M)\hspace{0.1cm}.    \] 

\medskip\noindent We consider the following vector field belonging to the tangent space of the cosphere bundle
\[ H = \e^{-\varphi}\bigg(\cos \theta \,    \frac{\d}{\d x_1}+ \sin \theta \,    \frac{\d}{\d x_2} +  a_{\varphi} \,   \frac{\d}{\d \theta} \bigg) \] 
where the coefficient $a_{\varphi}$ is given by 
\[ a_{\varphi} =-  \sin \theta \,  \d_{x_1}\varphi + \cos \theta \,  \d_{x_2}\varphi\hspace{0.1cm}.  \] 
The Hamiltonian vector field decomposes as 
\[ \frac{1}{2} H_p = p \e^{\varphi} b_{\varphi}  \varrho + \sqrt{p}  H \] 
where the coefficient $b_{\varphi}$ is given by
\[ b_{\varphi} = \cos \theta \,  \d_{x_1}\varphi + \sin \theta \,  \d_{x_2}\varphi\hspace{0.1cm}.  \] 
Now we set ourselves in the cosphere bundle and consider $H$, $V$ acting on functions on $S^*M$. We introduce  
\begin{align*}
      H_{\perp} = [H,V] 
\end{align*}
so that in coordinates 
\[ H_{\perp} = \e^{-\varphi}\bigg(\sin \theta \,    \frac{\d}{\d x_1} - \cos \theta \,    \frac{\d}{\d x_2} +  b_{\varphi} \,   \frac{\d}{\d \theta} \bigg)\hspace{0.1cm}.  \] 
The vector fields satisfy the following bracket algebra
\begin{align*}
      [V,H_{\perp}] = H, \quad [H_{\perp},H]=K V\hspace{0.1cm}.
\end{align*}

We now start with a solution $u \in H^2(S^*M;\mathbb{S}^{1})$ of the transport equation with real-valued potential $f \in C^{\infty}(S^*M)$
\begin{align}
\label{TransportEq}
      (H+if)u = 0\hspace{0.1cm}. 
\end{align}    
Consider the following functions in $H^1(S^*M)$ 
\begin{align*}
      \alpha = i\bar{u} H_{\perp} u,  \quad \beta =  i\bar{u} Vu\hspace{0.1cm}.
\end{align*}
From the fact that $u$ has values in the circle 
\[ |u|^2 =1 \hspace{0.1cm},\]
we get that $\alpha$ and $\beta$ are real-valued functions :
\[ \im \alpha =  \re (\bar{u}H_{\perp}u)= \frac{1}{2i} H_{\perp}(|u|^2)=0,   \quad \im \beta = \re  (\bar{u}Vu) = \frac{1}{2}V(|u|^2)=0\hspace{0.1cm}. \] 
Besides, from the transport equation \eqref{TransportEq}, we have 
\begin{align*}
      H\alpha &= iH\bar{u}\, H_{\perp}u+i\bar{u} \, HH_{\perp}u = if  \alpha - K \beta +  \bar{u}H_{\perp}(fu) \\ &=  if  \alpha - K \beta + H_{\perp}f - i f \alpha =  - K \beta + H_{\perp}f
\end{align*}
and similarly 
\begin{align*}
      H\beta &= iH\bar{u} \, Vu + i \bar{u} \, HVU = i f \beta + i\bar{u} \, VHu + \alpha \\
      &=if \beta +Vf -if \beta + \alpha= \alpha + Vf\hspace{0.1cm}.  
\end{align*}
Therefore we get the following system of equations on $(\alpha,\beta)$
\begin{align*}
     \begin{cases}
           H\alpha = - K \beta + H_{\perp}f \\
           H\beta = \alpha + Vf           
     \end{cases}
\end{align*}
which we can rewrite under the matrix form
\begin{align*}
       \begin{pmatrix}
             1 & -H \\ H & K 
       \end{pmatrix}
       \begin{pmatrix}
             \alpha \\ \beta 
       \end{pmatrix} 
       = 
       \begin{pmatrix}
       -Vf  \\  H_{\perp}f
       \end{pmatrix}\hspace{0.1cm}.        
\end{align*}
Having endowed the product space  $L^2(S^*M) \times L^2(S^*M)$  with the scalar product
\begin{align*}
     \left ( \begin{pmatrix} \phi_1 \\ \psi_1 \end{pmatrix} \right| \left. \begin{pmatrix} \phi_2 \\ \psi_2 \end{pmatrix} \right ) = ( \phi_1| \phi_2) + ( \psi_1| \psi_2)\hspace{0.1cm},
\end{align*}
where $(\, \cdot \, | \, \cdot \,)$ denotes the scalar product on $L^2(S^*M)$, the matrix operator 
\begin{align*}
       \mathcal{Q}
       =\begin{pmatrix}
             1 & -H \\ H & K 
       \end{pmatrix}
\end{align*} with domain $\mathop{\rm Dom}(H) \times \mathop{\rm Dom}(H) \subset L^2(S^*M) \times L^2(S^*M)$ defines a quadratic form which can be extended to $\mathop{\rm Dom}(H) \times  L^2(S^*M)$ 
\begin{align}
\label{qform}
     \left ( \mathcal{Q} \begin{pmatrix} \phi \\ \psi \end{pmatrix} \right| \left. \begin{pmatrix} \phi \\ \psi \end{pmatrix} \right ) 
     &= \|\phi\|^2 - 2 \re ( \phi|H\psi) + (K\psi|\psi) \notag  \\
     &= -\|H\psi\|^2+ (K\psi|\psi) + \|\phi-H\psi\|^2.
\end{align}

\begin{lem}
     The operator $\mathcal{Q}$ is self-adjoint 
     \[ \mathcal{Q}^* = \mathcal{Q}\hspace{0.1cm}. \] 
     and its kernel is trivial
     \[ \ker \mathcal {Q} = \{0\}\hspace{0.1cm}. \] 
     Therefore the polarisation of \eqref{qform} is a non-degenerate sesquilinear form. 
     Its restriction to the graph of $H$ is definite negative. 
\end{lem}
\begin{proof}
    The formal adjoint of $\mathcal{Q}$ reads
    \begin{align*}
        \begin{pmatrix}
            1 & -H \\ H & K 
        \end{pmatrix}^*
        = \begin{pmatrix}
            1 & H^* \\ -H^* & K 
        \end{pmatrix}
        =\begin{pmatrix}
            1 & -H \\ H & K 
        \end{pmatrix}\hspace{0.1cm},
    \end{align*}
    hence the operator $\mathcal{Q}$ is symmetric and since $(\phi_2,\psi_2) \in D(\mathcal{Q}^*)$ if and only if 
    \begin{align*}
     \left|\left ( \mathcal{Q} \begin{pmatrix} \phi_1 \\ \psi_1 \end{pmatrix} \right| \left. \begin{pmatrix} \phi_2 \\ \psi_2 \end{pmatrix} \right )\right|
     &= |(\phi_1|\phi_2) - (H\phi_1|\psi_2) - (H\psi_1,|\phi_2) + (K\psi_1|\psi_2)| \\
     &\leq C \sqrt{\|\phi_1\|^2 + \|\psi_1\|^2}
    \end{align*}
    for all $(\phi_1,\psi_1) \in \mathop{\rm Dom}(H) \times \mathop{\rm Dom}(H)$, choosing $\psi_1=0$ we deduce $\psi_2 \in \mathop{\rm Dom}(H)$ and choosing $\phi_1=0$ we deduce $\phi_2 \in \mathop{\rm Dom}(H)$. Hence $\mathcal{Q}$ is self-adjoint. 
    Besides, if $(\phi,\psi) \in \ker \mathcal{Q}$, we get from \eqref{qform}
    \[ \phi=H\psi, \quad \text{ and }\quad \|H\psi\|^2-(K\psi|\psi)=0\hspace{0.1cm}. \]
    If $(M,g)$ has no conjugate points this implies $\psi=0$ and therefore $\phi=0$.
    The negativity of the restriction of the quadratic form \eqref{qform} to the graph of $H$ comes from the fact that the absence of conjugate points implies $\|H\psi\|^2-(K\psi|\psi \geq 0$ for all $\psi \in \mathop{\rm Dom}(H)$ (see \cite{Hopf}). 
\end{proof}

\medskip\noindent Now we have 
\begin{align*}
     \left( \begin{pmatrix} -Vf \\  H_{\perp}f \end{pmatrix} \right| \left.\begin{pmatrix} \alpha \\ \beta\end{pmatrix} \right)
      &=  (f|V\alpha-H_{\perp}\beta)    
\end{align*}
and it remains to calculate the right-hand side term. 
\begin{lem}
\label{HperpBetaLem} 
$ H_{\perp}\beta = V\alpha -f\hspace{0.1cm}.$
\end{lem}
\begin{proof}
      Using the commutation rules of the vector fields $V$ and $H_{\perp}$, we get 
      \begin{align*}
             H_{\perp}\beta &= i (H_{\perp}\bar{u}) \, Vu + i \bar{u} H_{\perp}Vu = 
             i\bar{\alpha}\, \beta + i\bar{u}VH_{\perp}u - i \bar{u} Hu \\
             &= i\bar{\alpha} \, \beta  + V\alpha - i \alpha \, \beta - f 
      \end{align*}
      and this gives the expression since $\alpha$ and $\beta$ are real-valued.
\end{proof}
\begin{proof}[Proof of \textbf{Theorem \ref{RigidityTransportEq}} in the two-dimensional case]
Taking the scalar \\product of 
\begin{align*}
       \mathcal{Q}
       \begin{pmatrix}
             \alpha \\ \beta 
       \end{pmatrix} 
       = 
       \begin{pmatrix}
       -Vf  \\  H_{\perp}f
       \end{pmatrix}       
\end{align*}
with $(\alpha,\beta)$ and using \textbf{Lemma \ref{HperpBetaLem}}, we get 
\begin{align*}
     \mathcal{Q}
       \begin{pmatrix}
             \alpha \\ \beta 
       \end{pmatrix} 
       =
        \left( \begin{pmatrix} H_{\perp}f \\ -Vf \end{pmatrix}\right| \left.\begin{pmatrix} \alpha \\ \beta\end{pmatrix} \right)
      &= \|f\|^2.  
\end{align*}
This gives the identity 
\begin{align*}
      -\|H\beta\|^2  + (K \beta|\beta) + \|\alpha-H\beta\|^2 = \|f\|^2
\end{align*}
and since $\alpha-H\beta=-Vf$,
\begin{align}
\label{PestovId}
      -\|H\beta\|^2  + (K \beta|\beta) + \|Vf\|^2 = \|f\|^2.
\end{align}
Until now we have not used the special form of the function $f$
\[ f(x,\theta) = f_0(x) + f_1(x) \cos \theta + f_2(x) \sin \theta \] 
which implies that $f$ solves the equation
\[ V^2f + f = f_0\hspace{0.1cm}. \] 
This implies in turn
\[ \|f\|^2-\|Vf\|^2 = (V^2f+f|f) = (f_0|f) = \|f_0\|^2 \] 
and therefore (\ref{PestovId}) reads
\begin{align*}
      \underbrace{-\|H\beta\|^2  + (K \beta|\beta)}_{\leq 0} = \underbrace{\|f_0\|^2}_{\leq 0}.
\end{align*}
This finally gives $f_0=0$ and since the left-hand side is negative definite in the absence of conjugate points by the aforementioned classic result of Hopf \cite{Hopf},
\[ \beta=0\hspace{0.1cm}. \]
Hence $Vu=0$, which means that $\vartheta:=u$ does not depend on the fiber variable $\theta$.
The transport equation now takes the form
\[ f = i \bar{\vartheta} \, \dd \vartheta = i \frac{\dd \vartheta}{\vartheta} \] 
and we conclude that $f$ is closed :
\[ \dd f = 0\hspace{0.1cm}. \] 
Note that we have shown the existence of a function $\vartheta : M \to \mathbb{S}^1$ such that $f = i \bar{\vartheta} \, \dd \vartheta$. 
\end{proof}

\subsection{The higher-dimensional case}

The higher-dimensional case is a bit more tiresome as such a simple algebra of vector fields is not available.
The vertical subbundle $\mathscr{V}$   of the cotangent bundle is defined as 
\[ \mathscr{V}_{\varrho}= \ker \pi'(\varrho), \quad \varrho \in T^*M \] 
where $\pi'(\varrho) : T_{\varrho}(T^*M) \to T_{\pi(\varrho)}M$ is the differential at $\varrho$ of the projection $\pi : T^*M \to M$. 
The Levi-Civita connection induces the decomposition of the tangent space of $T^*M$
\[ T\left(T^*M\right) = \mathscr{H} \oplus \mathscr{V} \] 
into a horizontal subspace $\mathscr{H}$ and the vertical subspace. 
Those two subspaces are orthogonal for the Sasaki metric~\cite{Sasaki,T} on $T\left(T^*M\right)$
\begin{gather*}
   G = \sum_{j,k=1}^n g_{jk} \dd x_j \otimes \dd x_k  + \sum_{j,k=1}^n g^{jk} \mathrm{D}\xi_j \otimes \mathrm{D}\xi_k \\
    \text{where } \mathrm{D}\xi_j = \dd \xi_j - \sum_{\ell,m=1}^n \Gamma^{\ell}_{jm} \xi_{\ell} \, \dd x_{m}\hspace{0.1cm}. 
\end{gather*}
In local coordinates $(x,\xi) \in T^*M$,  the vertical subspace $\mathscr{V}_{(x,\xi)}$ is spanned by the derivatives
\[ \frac{\d}{\d \xi_j},  \quad 1 \leq j \leq n\hspace{0.1cm}, \]
and the horizontal subspace $\mathscr{H}_{(x,\xi)}$ by
\[ \nabla_j = \frac{\d}{\d x_j} + \sum_{k,\ell=1}^n \Gamma_{jk}^{\ell} \xi_{\ell} \frac{\d}{\d \xi_{k}} \quad 1 \leq j \leq n\hspace{0.1cm}. \] 
Raising indices, one obtains the horizontal component of the gradient with respect to the Sasaki metric :
\[ \nabla^j = \sum_{k=1}^n g^{jk}(x) \nabla_k =  \sum_{k=1}^n g^{jk}(x) \bigg( \frac{\d}{\d x_k} + \sum_{\ell,m=1}^n \Gamma_{k\ell}^{m}(x) \xi_{m} \frac{\d}{\d \xi_{\ell}}\bigg)\hspace{0.1cm}. \]   
The horizontal derivative $\nabla$ extends on one-forms in the following way
\begin{align*}
    (\nabla_j\alpha)_k = \nabla_j(\alpha_k) -  \sum_{\ell=1}^n \Gamma_{jk}^{\ell}(x) \alpha_{\ell}\hspace{0.1cm}.
\end{align*}
\begin{lem}
\label{lemStructEq}
     The commutator of two horizontal derivatives can be calculated in terms of the Riemannian curvature tensor :
     \[ [\nabla_j,\nabla_k] = \sum_{\ell,m=1}^n  {R^{\ell}}_{mjk}(x)  \xi_{\ell} \frac{\d}{\d \xi_{m}}\hspace{0.1cm}.\]  
\end{lem}
A proof of this well-known structural equation for horizontal derivatives on the tangent bundle can be found, for example, in \cite[Section 3.5]{Sh}, as well as a thorough analysis of the horizontal tensor algebra, but only tensors of degree one will be needed in our situation.  
Finally we introduce the scalar product 
\begin{align*}
   (\phi|\psi) = \int_{S^*M} \langle \phi,\psi\rangle_{G} \, \, \dd \omega
\end{align*}
where $\phi,\psi \in \Gamma(T(S^*M))$ and $\dd \omega$ denotes the Liouville measure on $S^*M$. 

\subsubsection{The Hamiltonian vector field}

As usual, let us note $p:(x,\xi)\mapsto \lvert\xi\rvert_{x}^{2}$ the kinetic energy and recall the expression of the associated Hamiltonian vector field $H_p$ on the cotangent bundle $T^*M$. 
\begin{lem}
    The Hamiltonian vector field can be expressed in terms of the contravariant derivatives as 
     \[ \frac{1}{2} H_p = \sum_{j=1}^n \xi_j \nabla^j =\sum_{j=1}^n g^{jk} \xi_j \nabla_k\hspace{0.1cm}. \] 
\end{lem}
\begin{proof}
      The compatibility of the Levi-Civita connection with the metric reads 
      \begin{align*}
           \frac{\d g^{jk}}{\d x_{\ell}} + \sum_{m=1}^n g^{jm} \Gamma^k_{\ell m} + g^{km} \Gamma^j_{\ell m} = 0\hspace{0.1cm}, 
      \end{align*}
      hence we have 
      \begin{align*}
            \frac{\d p}{\d x_{\ell}}  = - \sum_{j,k,m=1}^n \big(g^{jm} \Gamma^k_{\ell m} + g^{km} \Gamma^j_{\ell m}\big) \xi_j \xi_k = - 2 \sum_{j,k,m=1}^n g^{jm} \Gamma^k_{\ell m} \xi_j \xi_k
      \end{align*}
       and therefore 
       \begin{align*}
             H_p = 2\sum_{j,k=1}^n g^{jk} \xi_j \frac{\d}{\d x_k} + 2\sum_{j,k,\ell,m=1}^n g^{jm} \Gamma^k_{\ell m} \xi_j \xi_k \frac{\d}{\d \xi_{\ell}} =  2\sum_{j=1}^n \xi_j \nabla^j\hspace{0.1cm}.
       \end{align*} 
\end{proof}
\begin{rem}
      With the musical isomorphisms, 
      \[ v = \xi^{\sharp}, \quad \xi = v^{\flat}, \quad  v^j = \sum_{k=1}^n g^{jk} \xi_j \hspace{0.1cm},\] 
      making the change of variables $y=x$, $v=\xi^{\sharp}$,
       it is easy to recover the generator of the geodesic flow on the tangent bundle
      \begin{align*}
           X = \flat_*H = \sum_{j=1}^n v_j \bigg(\frac{\d}{\d x_j} - \sum_{k,\ell=1}^n  \Gamma_{jk}^{\ell}(x)v_k \frac{\d}{\d v_{\ell}}   \bigg)\hspace{0.1cm}.
\end{align*}
      The literature in geometric inverse problems tends to make use of the vector field $X$ and to be set in the tangent space (see for instance \cite{PatSalUhlbook}). Our natural preference here is rather to work on the cotangent bundle, as mentioned at the end of Subsection \ref{Sect 1.2}. 
\end{rem}

\subsubsection{Algebra on the cosphere bundle}

Let $u \in C^{\infty}(S^*M)$ be a smooth function on the cosphere bundle, we denote by
\begin{align*}
     U_d(x,\xi) = r^d(x,\xi) \, u\bigg(x,\frac{\xi}{r(x,\xi)}\bigg), \quad (x,\xi) \in T^*M \setminus 0 
\end{align*}
the homogeneous extension of degree $d \in \R$ of $u$,  which is a smooth function on $T^*M \setminus 0$, and we define
\[ V_ju =\sum_{k=1}^n g_{jk}\frac{\d U_0}{\d \xi_k}\bigg|_{S^*M}. \]
Since 
\[ \frac{\d r}{\d \xi_j}\bigg|_{S^*M} = \theta^j,  \quad \theta^{\sharp}=(\theta^1,\dots,\theta^n) \hspace{0.1cm},\] 

\medskip\noindent we also have the relation
\begin{align*}
    V_ju(x,\theta) &= \sum_{k=1}^n g_{jk} \frac{\d U_d}{\d \xi_k}(x,\theta) - d \theta_j u(x,\theta)\hspace{0.1cm},
\end{align*}
which is useful in finding formulas involving $V_j$. 
\begin{rem}
      If we consider the projection
      \begin{align*}
           \pi : S^*M \to M \hspace{0.1cm},
      \end{align*}
      vertical vector fields are those who belong to  $\ker \dd \pi$.  
      The vectors fields $V_j$ are vertical.  Note that in the two-dimensional case,  with the slight abuse of notations $(\theta_1,\theta_2)=\e^{\varphi}(\cos \theta,\sin \theta) \in S^*M$ with in this case $\theta \in \R$, we had
      \[ V_1 = -\e^{-\varphi} \sin \theta \frac{\d}{\d \theta}, \quad V_2=\e^{-\varphi} \cos \theta \frac{\d}{\d \theta} \] 
      and we were considering the vector field
      \[ V = \frac{\d}{\d \theta} = -\sin \theta  \,V_1 + \cos \theta \, V_2\hspace{0.1cm}.  \] 
\end{rem}
Contrary to the two-dimensional case where only one vertical vector field could be used, the vertical vector fields $V_j$ are not independent and  Euler's relation on homogeneous functions provide the following linking relation
\begin{align}
\label{EulerCsq}
     \sum_{j,k=1}^n  g^{jk}(x) \theta_j  V_k =  0\hspace{0.1cm}.
\end{align}
The following commutation relations hold
\begin{align*}
     [V_j,\theta_k]  = g_{jk}-\theta_j \, \theta_k, \quad  [V_j,V_k] &=\theta_j V_k-\theta_k V_j\hspace{0.1cm}.
\end{align*}
\begin{lem}
 $V_j^* = -V_j+ (n-1)\theta_j$
\end{lem}
\begin{proof}
      Let $\tilde{\chi} \in C^{\infty}(\R)$ be a real-valued function with compact support in $[1/2,2]$, which equals one in $[1/4,3/4]$ and satisfying 
      \[ \int_0^{\infty} r^{n-1} \tilde{\chi}(r) \, \dd r = 1 \hspace{0.1cm},\]
      so that with $\chi(x,\xi)=\tilde{\chi}(|\xi|_g)$ we have 
      \begin{align*}
         \int_{T^*M} \chi(x,\xi) U_0(x,\xi) \, \dd \lambda &= 
         \int_0^{\infty} \int_{S^*M} r^{n-1}\chi(r) u(x,\theta) \, \dd r \wedge \dd \omega \\
         &= \int_{S^*M} u(x,\theta) \dd \omega.
      \end{align*}
      We denote $\Xi_j=\sum_{k=1}^n g_{jk} \d / \d \xi_k$ and 
      using the former integral we get 
      \begin{align*}
             \int_{S^*M} (V_j\phi) \,  \overline{\psi} \, \dd \omega &= \int_{T^*M} \chi \, (\Xi_j \phi_0) \,  \overline{\psi_1} \, \dd \lambda \\
             &= -\int_{T^*M} \chi \, \phi_0 \,  \overline{\Xi_j\psi_1} \, \dd \lambda
             -\int_{T^*M} (\Xi_j\chi) \, \phi_0 \,  \overline{\psi_1} \, \dd \lambda.           
      \end{align*}
      Since we have 
      $
          \Xi_j \chi(x,\xi) = \theta_j \tilde{\chi}'(r)
      \hspace{0.1cm},$
      we recover 
      \begin{align*}
         \int_{T^*M} (\Xi_j\chi) \, \phi_0 \,  \overline{\psi_1} \, \dd \lambda =
         \int_{S^*M} \theta_j \phi \, \overline{\psi} \bigg(\underbrace{\int_0^{\infty} \tilde{\chi}'(r) r^n \, \dd r}_{=-n}\bigg) \, \dd \omega.
      \end{align*}
      We also have 
      $
          (\Xi_j \psi_1)(x,\theta) = V_j\psi(x,\theta)+ \theta_j \psi(x,\theta)
      $
      hence 
      \begin{align*}
             \int_{T^*M} \chi \, \phi_0 \,  \overline{\Xi_j\psi_1} \, \dd \lambda
             = \int_{S^*M} \phi \,  \overline{V_j \psi} \, \dd \omega 
             + \int_{S^*M} \phi \,  \theta_j\overline{\psi} \, \dd \omega. 
      \end{align*}
      In conclusion,
      \begin{align*}
             \int_{S^*M} (V_j\phi) \,  \overline{\psi} \, \dd \omega &=-\int_{S^*M} \phi \,  \overline{V_j \psi} \, \dd \omega + (n-1) \int_{S^*M} \phi \,  \theta_j\overline{\psi} \, \dd \omega          
      \end{align*}
      
      \medskip\noindent and we get that the adjoint of $V_j$ is the claimed one. 
\end{proof}
\medskip\noindent Similarly to the vectors $V_j$, since 
\[ \nabla_j r = \frac{1}{2r}\sum_{\ell,m=1}^n \big(\d_{x_{\ell}}g_{jm}-\d_{x_m}g_{j\ell}\big) \xi^{\ell} \xi^m 
= 0\hspace{0.1cm},  \] 
we have 
\[ \nabla_j U_{d} = \nabla_j U_0\hspace{0.1cm}, \]

\medskip\noindent and it makes sense to define the covariant derivatives on the cosphere bundle by
\[ \nabla_ju = (\nabla_j U_0)\big|_{S^*M} = (\nabla_j U_d)\big|_{S^*M}\hspace{0.1cm}. \] 

\medskip\noindent They commute with the variable $\theta$ 
\begin{align*}
     [\nabla,\theta] = \nabla(\theta) = 0
\end{align*}
and 
\begin{align*}
    [ \nabla,V] = \sum_{\ell=1}^n \big(\d_{x_j}g_{k\ell}-\d_{\ell}g_{jk}\big)\frac{\d}{\d \xi_{\ell}} = 
    \sum_{\ell=1}^n \big(\Gamma_{\ell jk}-\Gamma_{jk\ell}\big) \frac{\d}{\d \xi_{\ell}}\hspace{0.1cm}.
\end{align*}

Now we are ready to develop the algebra of commuting operators as in the two-dimensional case. 
All of this is of course very standard but for the reader's convenience, we will briefly outline it here. 
Recall that we are interested in the operator 
\[  H = \sum_{j=1}^n \theta_j \nabla^j = \sum_{j=1}^n \theta^j \nabla_j\hspace{0.1cm}. \] 
From the properties of the vertical vector fields,
\begin{align*}
       [\theta_k \nabla^k,V_j] &= \theta_k \underbrace{[\nabla^k,V_j]}_{=0} + [\theta_k,V_j]\nabla^k = -g_{jk} \nabla^k + \theta_j \theta_k \nabla^k\hspace{0.1cm}, 
\end{align*}  
we get 
\[ [H,V] = -\nabla + \theta H \]

\medskip\noindent and this leads to introduce the following vector-valued operator 
\[ H_{\perp}u =(-\nabla_1u,\dots,-\nabla_nu) \]
for $u \in H^1(S^*M)$. 
The structure equation from \textbf{Lemma \ref{lemStructEq}} gives 
\begin{align*}
       [\theta^k\nabla_k,\nabla_j] &= \theta^k[\nabla_k,\nabla_j] + \underbrace{[\theta^k,\nabla_j]}_{=0} \nabla_k \\
       &= \sum_{\ell,m=1}^n {R^{\ell}}_{mkj} \theta^k \theta_{\ell}\frac{\d}{\partial\xi_m}
       = \sum_{\ell,m=1}^n {R_{kj\ell}}^m \theta^k \theta^{\ell} V_m
\end{align*}  
and this provides the following commutator 
\begin{align*}
     [H,\nabla_j] =  \sum_{k=1}^n [\theta_k \nabla^k,\nabla_j] =  \sum_{k,\ell,m=1}^n {R_{kj\ell}}^m \theta^k \theta^{\ell}V_m\hspace{0.1cm}. 
\end{align*}
We finish by introducing the endomorphism  
\[ (\hat{R} \phi)_j = \sum_{k,\ell,m=1}^n {R_{jk\ell}}^m \theta^k \theta^l \phi_m = R(\theta,\phi)\theta \] 
based on the Riemann curvature tensor. Having recalled the algebra of $H, H_{\perp}, V_j$,  we can now proceed to the analysis of the transport equation with potential. 

\subsubsection{The equations}

Let $u \in H^2(S^*M;\mathbb{S}^{1})$ be a solution of the transport equation
\[ (H+if)u = 0\hspace{0.1cm}. \]
As in the two-dimensional case, we define the following vectors
\[ \alpha := i \bar{u}H_{\perp}u \in \mathscr{H},  \quad \beta :=  i \bar{u}Vu  \in \mathscr{V}\] 
which have real-valued components since
\[ 0 = V_j(|u|^2)= 2\re (\bar{u}V_ju)\quad, \quad 0=\nabla_j(|u|^2)=2\re(\bar{u}\nabla_ju)\hspace{0.1cm}.  \]
We will use matrix notations 
\begin{align*}
      \begin{pmatrix} \alpha \\ \beta \end{pmatrix} \in \mathscr{H} \oplus \mathscr{V} = T(S^*M) 
\end{align*}
for vectors with horizontal and vertical components and consider the following matrix operator 
\begin{align*}
       \mathcal{Q}
       =\begin{pmatrix}
             {\rm I}_n  & -H \\ H & \hat{R}
       \end{pmatrix}.
\end{align*}
We now calculate 
\begin{align*}
      H\alpha &= i(H\bar{u}) H_{\perp}u + i\bar{u}H(H_{\perp}u) 
      = if \alpha  + i \bar{u}H_{\perp}Hu - i \bar{u} \hat{R} Vu \\
      &= if \alpha + H_{\perp}f -if \alpha  + \hat{R}\alpha 
      = H_{\perp}\alpha + \hat{R}\alpha\hspace{0.1cm},
\end{align*}
as well as 
\begin{align*}
      H\beta &= i(H\bar{u})Vu + i\bar{u}HVu = if \beta + i\bar{u}VHu + \alpha + \theta f \\
      &= if\alpha - i f\alpha + Vf + \alpha + \theta f = \alpha + (V+\theta)f
\end{align*}
and this leads to the following system 
\begin{align}
\label{Qsystem}
      \mathcal{Q} \begin{pmatrix}
      \alpha \\ \beta 
      \end{pmatrix}
      = 
      \begin{pmatrix}
             -(V+\theta)f \\ H_{\perp} f
      \end{pmatrix}.
\end{align}
Note that 
\[ \langle \theta,\alpha \rangle =-i\bar{u}  \sum_{j,k=1}^n g^{jk} \theta_j \nabla_ku =- i \bar{u}Hu=-f \] 
and therefore 
\[ (\theta f|\alpha) = -(f|f) = - \|f\|^2. \] 
Since the vertical and horizontal bundles are orthogonal for the Sasaki metric, we have 
\begin{align*}
     \left (\begin{pmatrix} \phi_1 \\ \psi_1 \end{pmatrix} \right| \left. \begin{pmatrix} \phi_2 \\ \psi_2 \end{pmatrix} \right ) &= 
     (\phi_1|\phi_2) + (\psi_1|\psi_2), \quad \phi_1,\phi_2 \in \mathscr{H}, \, \psi_1,\psi_2 \in \mathscr{V}. 
\end{align*}
The matrix operator $\mathcal{Q}$ is associated with the quadratic form 
\begin{align*}
     \left ( \mathcal{Q} \begin{pmatrix} \phi \\ \psi \end{pmatrix} \right| \left. \begin{pmatrix} \phi \\ \psi \end{pmatrix} \right ) &= \|\phi\|^2 + 2 \re ( H\phi|\psi) + (\hat{R}\psi| \psi)    \\
     &= -\|H\psi\|^2+ (\hat{R}\psi|\psi) + \|H\psi-\phi\|^2
\end{align*}
and, endowed with the domain $\mathop{\rm Dom }(H) \times \mathop{\rm Dom }(H)$, is self-adjoint as in the two-dimensional case. Besides, we have  
\begin{align*}
     \left( \begin{pmatrix} -(V+\theta)f \\ H_{\perp}f \end{pmatrix} \right| \left.\begin{pmatrix} \alpha \\ \beta\end{pmatrix} \right)
      &= -(Vf|\alpha) + \|f\|^2 + (H_{\perp}f|\beta) 
\end{align*}
which we now calculate. 
\begin{lem}
    $-(Vf|\alpha)+(H_{\perp}f|\beta) = (n-1) \|f\|^2.$
\end{lem}
\begin{proof}
    We start by computing the adjoint of $H_{\perp}$   
     \begin{align*}
         H_{\perp}^* = [V,H]^* = -[V+(n-1)\theta,H]=-H_{\perp} + \underbrace{(n-1) H(\theta)}_{=0}\hspace{0.1cm},
     \end{align*}
     which implies that 
     \begin{align*}
        -(Vf|\alpha)+(H_{\perp}f|\beta) &= (f|V\alpha) - (n-1) (\theta f|\alpha) - (f|H_{\perp}\beta) \\
        &= (f|V\alpha-H_{\perp}\beta) + (n-1) \|f\|^2.
     \end{align*}
     The computation 
     \begin{align*}
        (\nabla_j\beta)_{k}  &= \nabla_j(\beta_k) - \sum_{\ell=1}^n \Gamma^{\ell}_{jk} \beta_{\ell} \\
        &= i(\nabla_j\bar{u}) V_ku + i\bar{u}V_k(\nabla_ju)+i\bar{u}\bigg(\underbrace{[\nabla_j,V_k] u - \sum_{\ell=1}^n \Gamma^{\ell}_{jk} V_{\ell}u}_{=0} \bigg)\\
        &=-i \alpha_j \beta_k  - \bar{u}V_k(u \alpha_j)
        = -V_k\alpha_j
     \end{align*}
     
     \medskip\noindent leads to $V\alpha -H_{\perp}\beta=0$ and allows to conclude. 
\end{proof}

In all these preliminary steps, we haven't used the affine form of the function $f$ in the fiber which allows to calculate the following $L^{2}$-norms :
\begin{lem}
\label{SpecialformCsq}
     Suppose that $f=f_0+f_1$ is the sum of a function $f_0$ and a one form $f_1$ on $M$. Then 
     \begin{align*}
    \|f\|^2 = \|f_0\|^2 + \|f_1\|^2, \quad  
    \|Vf\|^2 = \|Vf_1\|^2 =(n-1) \|f_1\|^2.
\end{align*}
\end{lem}
\begin{proof}
     In a coordinate chart $U$, the one-form $f_1$ takes the form 
     \[ f_1(x,\theta) = \sum_{j=1}^n a^j(x) \theta_j, \quad \theta \in S_x^*(U) \]
     and the first identity being clear, we concentrate on the second one; since 
    \begin{align*}
        V_kf_1 = a_k -f_1 \theta_k\hspace{0.1cm},
    \end{align*}
    we get on $S^*(U)$
    \begin{align*}
        |Vf_1|^2 = |a|^2 - f_1^2.
    \end{align*}
    Integration of a quadratic form on the hyperpshere provides a multiple of the trace 
    \begin{align*}
        \int_{S_x(U)} \sum_{j,k=1}^n a^j(x) a_k(x) \theta_j \theta^k \, \dd \theta = \frac{\Omega_n}{n} \bigg(\sum_{j=1}^n a_j(x)a^j(x)\bigg) \sqrt{\det g(x)} \hspace{0.1cm},
    \end{align*}
    where $\Omega_n=2\pi^{n/2}\Gamma(n/2)$ is the volume of the hypersphere,
    then integration over the chart $U$ leads to 
    \[ \|f_1\|_{L^2(U)}^2  = \frac{\Omega_n}{n} \|a\|_{L^2(U)}^2 = \frac{1}{n-1} \|Vf_1\|_{L^2(U)}^2 \]
    and the lemma is proved. 
\end{proof}

\begin{proof}[Proof of \textbf{Theorem \ref{RigidityTransportEq}} in all dimensions $n \geq 2$]
Taking the scalar product of \eqref{Qsystem} with $(\alpha,\beta)^{T}$ leads to 
\begin{align*}
    - \|H\beta\|^2+ (\hat{R}\beta,\beta) + \lVert\underbrace{\alpha-H\beta}_{=-(V+\theta)f}\rVert^2
    = n\|f\|^2
\end{align*}
and since 
\[ \|(V+\theta)f\|^2 = \|Vf\|^2 + \|f\|^2 + 2 \re \underbrace{(\theta f|Vf)}_{=0} \hspace{0.1cm},\]
one obtains 
\begin{align*}
    - \|H\beta\|^2+ (\hat{R}\beta|\beta) 
    = (n-1)\|f\|^2 - \|Vf\|^2.
\end{align*}
Now using the special form of $f$, by \textbf{Lemma \ref{SpecialformCsq}} we get 
\begin{align*}
    - \|H\beta\|^2+ (\hat{R}\beta|\beta) 
    = (n-1)\|f_0\|^2 \hspace{0.1cm}.
\end{align*}
In the absence of conjugate points, by the generalisation of Hopf's result \cite{Hopf} to higher-dimensional manifolds, the left-hand side term is non-positive definite, and therefore 
\[ f_0 =0\hspace{0.1cm},\hspace{0.1cm}\beta = 0\hspace{0.1cm}. \]
Consequently, $Vu=0$ and $u$ is a function $\vartheta : M \to \mathbb{S}^1$ so that 
\[ f = f_1 = -i \bar{u} Hu=-i \bar{\vartheta}\dd \vartheta\hspace{0.1cm}. \]
The theorem is proved.
\end{proof}

\section{The inverse problem for magnetic Schr\"odinger operators}\label{Section 4}

Here we show how to derive \textbf{Theorem \ref{Th 1.1}} as a direct consequence of the previous section. In order to exploit the principal invariants of the wave trace, we first recall that the subprincipal symbol of the magnetic Schrödinger operator $P:=P_{a,q}$ is given exactly by the magnetic potential :

\medskip\noindent The (positive) Laplace-Beltrami operator also reads in coordinates 
\[ \delta d=-\Delta = \sum_{j,k=1}^n g^{jk}D_jD_k + i \sum_{j,k=1}^n g^{jk} \Gamma_{jk}^{\ell} D_{\ell}\hspace{0.1cm}, \]
therefore, the homogeneous part of degree one of the full symbol in these local coordinates is given by
\[    p_1(x,\xi) = 2 \langle a(x),\xi \rangle + i \sum_{j,k=1}^n g^{jk}(x) \Gamma_{jk}^{\ell}(x) \xi_{\ell}\hspace{0.1cm}. \] 
Of course this is not an invariantly defined function on the cotangent bundle  but the principal symbol is 
\[ \sigma_P(x,\xi) = p(x,\xi) = |\xi|_{x}^2 = \sum_{j,k=1}^n g^{jk}(x) \xi_j \xi_k\hspace{0.1cm}, \] 
as well as the subprincipal symbol 
\begin{align*}
     \sub (P)(x,\xi) &= p_1(x,\xi) - \frac{1}{2i} \sum_{j=1}^n \frac{\d^2p}{\d x_j \d \xi_j} \\
     &=  2\langle a(x),\xi \rangle + \underbrace{i \sum_{j,k=1}^n g^{jk}(x) \Gamma_{jk}^{\ell}(x) \xi_{\ell}
     -i  \sum_{j=1}^n \frac{\d g^{jk}}{\d x_j}(x) \xi_j }_{=0}\hspace{0.1cm}.
\end{align*}

\subsection{Gauge invariance}
\label{SubSecGauge}

There is a well-known and natural gauge invariance in the spectral inverse problem for the magnetic Schr\"odinger operator $P:=P_{a,q}$ , namely if $\vartheta \in C^{\infty}(M;\mathbb{S}^1 )$ is a smooth map with values in the unit circle $\mathbb{S}^1 \subset \C$, then 
\[ \widetilde{P} =  \bar{\vartheta}P\vartheta \]
is a magnetic Schr\"odinger operator with electric potential $q$ and magnetic potential 
\[ \tilde{a} = a-i\bar{\vartheta}\dd \vartheta \]

\medskip\noindent and obviously with the same spectral data as $P$. 
Indeed, writing $\dd_a:=\dd+ia$ the magnetic exterior derivative, we have $P = \dd_a^*\dd_a^{\vphantom{*}} + q$ and 
\[ \bar{\vartheta}P\vartheta =  \bar{\vartheta}\dd_a^*\vartheta  \bar{\vartheta}\dd_a^{\vphantom{x}}\vartheta =  (\bar{\vartheta}\dd_a\vartheta)^*  \bar{\vartheta}\dd_a\vartheta\hspace{0.1cm},  \]  
where the conjugated factor is given by 
\[  \bar{\vartheta}\dd_a\vartheta = \dd +i(a-i\bar{\vartheta}\dd \vartheta)\hspace{0.1cm}. \] 
Note that the one-form $i \bar{\vartheta} \dd \vartheta$ is real-valued since 
\[ 0 = \dd |\vartheta|^2 = 2 \re (\bar{\vartheta} \dd \vartheta)\hspace{0.1cm}.\]

\begin{rem}   
     If $M$ is simply connected, then $p:t \mapsto \e^{it}$   
     is a covering map and $\vartheta_*(\pi_1(M)) = \{0\} = p_*(\pi_1(\R))$.  
     Consequently, one can lift the function $\vartheta$, in other words, there exists $\psi \in C^{\infty}(M;\R)$ a smooth real-valued function such that $\vartheta = \e^{i \psi}$. 
     
     \medskip\noindent However, an Anosov manifold is never simply connected (see \cite{Kl}) and the function $\vartheta$ cannot be lifted. In that case, as already mentioned at the end of the introduction (Section \ref{Section 1.1}), the gauge is related to the so-called \textit{Aharonov-Bohm effect}, where one tries to recover the subprincipal terms $a-\tilde{a}$ from the spectral data on domains with nontrivial topology.
\end{rem} 
In other words, this is an obstruction to identifiability of the magnetic potential from the spectrum, and a gauge invariance that has to be taken into account in the spectral inverse problems. 
Nevertheless, the gauge is closed 
\[ \dd( i\bar{\vartheta} \,  \dd \vartheta) = i \dd \bigg( \frac{\dd \vartheta}{\vartheta} \bigg) = 0 \]  
and is therefore not an obstruction to identifiability of the magnetic field $b =\dd a$. 

\medskip\noindent Conversely, one can show that two magnetic potentials with the same magnetic fields are gauge-equivalent under some additional cohomological assumption.  The following lemma recalls a suitable version of \cite[Proposition 3.1]{Shi} :
\begin{lem}\label{Lemma 4.2}
     Let $(M,g)$ be a closed Riemannian manifold, and $a, \tilde{a} \in \Omega^1(M)$ be two magnetic potentials. The following statements are equivalent :
     \begin{enumerate}[$(i)$]
         \item the two magnetic potentials are gauge-equivalent, i.e. there exists $\vartheta \in C^{\infty}(M;\mathbb{S}^1 )$ such that 
     \[ \tilde{a} = a - i \bar{\vartheta} \dd \vartheta\hspace{0.1cm}. \]
         \item the two magnetic fields are equal $\dd a = \dd\tilde{a}$ and 
         \[ \int_{\gamma} \pi_{1}^{*}\left(a-\tilde{a}\right) \in 2\pi \mathbb{Z} \] 
     for any closed orbit $\gamma$ of the geodesic flow of $M$.
     \end{enumerate}
\end{lem}
\medskip\noindent This result is a direct consequence of \cite[Proposition 3.1]{Shi} and the fact that there is always at least one geodesic in each free homotopy class of a closed Riemannian manifold.

\subsection{Recovery of the magnetic and electric potentials}

Up to this gauge, we are now able to recover the magnetic potential from the spectrum of the magnetic Schrödinger operator $P:=P_{a,q}$, in the case where the boundary is Anosov with simple length spectrum : 

\begin{proof}[Proof of \textbf{Theorem \ref{Th 1.1}}]
Since $P_{a,q}$ and $P_{\tilde{a},\tilde{q}}$ have the same spectrum, the principal wave invariants (cf. \textbf{Corollary \ref{Cor 2.2}}) gives that for any periodic geodesic $\gamma$ of $M$ :
\[ \int_{\gamma}\pi_{1}^{*}\left(\tilde{a}-a\right) \in 2\pi\Z \hspace{0.1cm}.\] 
Then by \textbf{Proposition \ref{Th 3.1}}, there exists a solution  $u \in C^{\infty}(S^*M;\mathbb{S}^{1})$ of the transport equation with potential $f:=\pi_{1}^{*}\left(\tilde{a}-a\right)$ :
     \[ Hu+ifu = 0 \hspace{0.1cm}.\]
Now using \textbf{Theorem \ref{RigidityTransportEq}}, we get that the one-form $\pi_{1}^{*}\left(\tilde{a}-a\right)$ is closed, $\vartheta:=u$ does not depend on the fiber variable, and by \textbf{Lemma \ref{Lemma 4.2}}, satisfies the desired identity : 
\[ \tilde{a} = a-i\bar{\vartheta}\dd \vartheta \hspace{0.1cm}.\] 
The operator
\[ \bar{\vartheta}P_{a,q}\vartheta = P_{a-i\bar{\vartheta}\dd \vartheta,q} =P_{\tilde{a},q} \]
has the same spectrum than $P_{\tilde{a},\tilde{q}}\hspace{0.1cm},$ therefore a variation of the argument of \cite{G2} (see also Section \ref{Section 5.3} and \cite{Flo} in the case of the Steklov operator) reduces the spectral inverse problem to that of \cite{GK1} and gives
\[ \int_{\gamma} \left(\tilde{q}-q\right) = 0 \]
for all periodic geodesic $\gamma$. 
The injectivity of the geodesic X-ray transform on functions implies $\tilde{q}=q.$ 
\end{proof}

\section{Magnetic Steklov inverse problem}\label{Section 5}
In this section, we prove \textbf{Theorem \ref{Th 1.2}} and \textbf{Theorem \ref{Th 1.3}} on the recovery of a magnetic and electric potential from the spectrum of the associated magnetic DN map.
Unlike the case of the magnetic Laplacian treated above, we are able to extract more information from the spectrum of the magnetic DN map, namely that it determines the jet at the boundary of the magnetic field and of the electric potential. This is done using the same approach as in \cite[Section IV]{Flo}, that is, by induction on the order of derivation, using at each step the trace formula of \textbf{Theorem \ref{Th 2.3}}. We will therefore need to identify, in suitable coordinates, the homogeneous term at any order in the full symbol of $$Q:=\Lambda_{a,q}-\Lambda_{\widetilde{a},\widetilde{q}}\hspace{0.1cm}.$$ 

\medskip\noindent In order to do that, we will consider \textit{boundary normal coordinates} and exploit \cite[Section 8]{DSFKSU} or, more generally, \cite[Section 3]{Cek} for the expression of the full symbol of the magnetic DN map in these coordinates. \\Recall that if $x':=(x_{1},...,x_{n-1})$ denotes any local coordinates on the boundary near $p\in\partial N$, one can define the associated boundary normal coordinates $(x',x_{n})$ at $p$ by considering $x_{n}$ to be the distance to the boundary along unit speed geodesics normal to $M:=\partial N$. See, for example, \cite[Section 1]{LU} for more details. It follows that in these coordinates, $x_{n}>0$ in $N$, the boundary $M$ is locally characterized by $x_{n}=0\hspace{0.1cm},$ and the metric has the form 

$$g=\sum_{1\leq \alpha,\beta\leq n-1}g_{\alpha\beta}\left(x\right) \dd x_{\alpha}\otimes \dd x_{\beta}+ \dd x_{n}\otimes \dd x_{n}\hspace{0.1cm}.$$

\medskip\noindent Note that the aforementioned computation of the full symbol of the magnetic DN map in boundary normal coordinates is made possible by putting the magnetic potential in a suitable gauge and making a conformal normalization of the metric. By nature, these transformations preserves the DN map, so in this section, unless otherwise stated, we will therefore always consider normalized data and note that in our problem it does not reduce the generality. We refer to \cite[Section 8]{DSFKSU} and to \cite[Section 3]{Cek} for details. 
\begin{rem}\label{Rk 5.1}
More precisely, the conformal normalization, which affects both the metric and the electric potential, is the same for each pair of data $(a,q)$ and $(\tilde{a},\tilde{q})$ since the metric is fixed in our problem. Therefore, recovering the jet at the boundary of a normalized electric potential also implies recovering the jet at the boundary of the initial one. However, this is not true for the magnetic potential, because for each pair $a$ and $\tilde{a}$, we must choose different gauge transformations a priori. On the other hand, since the gauge terms are closed, this does not affect our main results on the recovery of magnetic fields.    
\end{rem}

\subsection{Gauge invariance}\label{Subsection 5.1}

As in the case of the Laplacian, there is a natural gauge invariance that is completely analogue, namely if $\vartheta \in C^{\infty}(N ; \mathbb{S}^1 )$ is a smooth map with values in the unit circle $\mathbb{S}^1 \subset \C$, then 
\[ \tilde{\Lambda}:=\restriction{\bar{\vartheta}}{M}\Lambda_{a,q}\restriction{\vartheta}{M} \]
is a magnetic DN map with electric potential $q$ and magnetic potential 
\[ \tilde{a} = a-i\bar{\vartheta}\dd \vartheta \] 
and obviously with the same spectrum as $\Lambda_{a,q}$. 
Indeed, a direct calculation shows that for any $f\in H^{1/2}(M)$,

$$\begin{aligned}[t]
\widetilde{\Lambda}f &= \restriction{\left(\dd_{a}u+\bar{\vartheta}(\dd\vartheta)u\right)}{M}(\nu)\\
&= \restriction{\left(\dd u+ i(a-i\bar{\vartheta}\dd\vartheta) u\right)}{M}(\nu)\\
&= \restriction{\dd_{\tilde{a}}u}{M}(\nu)
\end{aligned}$$

\medskip\noindent where $u$ denotes the unique solution in $H^{1}(N)$ of the Dirichlet problem

$$\left\lbrace
\begin{aligned}
P_{\tilde{a},q}u &= 0 \\
\restriction{u}{\partial N} &= f \\
\end{aligned}\right.\hspace{0.2cm}\cdot$$

\medskip\noindent Therefore, from the magnetic Steklov spectrum one can only hope, at best, to recover the magnetic potential up to gauge, i.e. up to a term of the form $i\bar{\vartheta}\dd\vartheta$ where $\vartheta \in C^{\infty}(N;\mathbb{S}^1 )$. In the following, we are only able to do it on the boundary provided the latter is Anosov with simple length spectrum, this is the object of the next part. Since the gauge is closed (cf. Subsection \ref{SubSecGauge}), it is therefore not an obstruction to the unique identifiability of the magnetic field at the boundary, as already mentioned. To conclude this section, note that this gauge transforms the subprincipal symbol in a simple way : 

\begin{lem}\label{Lemma 5.2}
Let $\vartheta\in C^{\infty}(M,\mathbb{S}^{1})$ and $\Lambda$ be a magnetic DN map on $N$. Then the following identity holds on the unit cotangent bundle $S^{*}M$ :
$$\sub(\bar{\vartheta}\Lambda\vartheta)=\sub(\Lambda)-i\bar{\vartheta}\pi_{1}^{*}\dd\vartheta\hspace{0.1cm}.$$ 
\end{lem}
\begin{proof}
If $P$ and $Q$ are pseudodifferential operators on the closed manifold $M$, it is well-known that the subprincipal symbol of $PQ$ is given by
$$\sub(PQ)=\sub(P)\cdot\sigma_{Q}+\sigma_{P}\cdot\sub(Q)+\frac{1}{2i}\lbrace\sigma_{P},\sigma_{Q}\rbrace\hspace{0.1cm},$$
where $\lbrace\cdot ,\cdot\rbrace$ denotes the Poisson bracket. Using this formula, it is easy to see that
$$\sub(\bar{\vartheta}\Lambda\vartheta)=\sub(\Lambda)+\frac{1}{2i}\left(\bar{\vartheta}\lbrace \sigma_{\Lambda},\vartheta\rbrace+\vartheta\lbrace\bar{\vartheta}, \sigma_{\Lambda}\rbrace\right)\hspace{0.1cm}.$$
Since $\vartheta$ does not depend on the fiber variable and has modulus $1$, we show that in any coordinate chart $(x,\xi)$ on the unit cotangent bundle $S^{*}M$ : 
\begin{align*}
\frac{1}{2i}\left(\bar{\vartheta}\lbrace \sigma_{\Lambda},\vartheta\rbrace+\vartheta\lbrace\bar{\vartheta}, \sigma_{\Lambda}\rbrace\right)
&=-i\bar{\vartheta}\sum_{j=1}^{n-1}\frac{\partial\sigma_{\Lambda}}{\partial \xi_{j}}\frac{\partial\vartheta}{\partial x_{j}}\\
&= -i\bar{\vartheta}\sum_{j,k=1}^{n-1}g^{kj}\frac{\partial\vartheta}{\partial x_{j}}\dd x_{k}\\
&=-i\bar{\vartheta}\pi_{1}^{*}\dd\vartheta \hspace{0.1cm}.
\end{align*}
This proves the identity on subprincipal symbols.
\end{proof}
\begin{rem}\label{Rem 5.3}
If $\vartheta\in C^{\infty}(M,\mathbb{S}^{1})$ and $\Lambda$ is a magnetic DN map on $N$, then $\bar{\vartheta}\Lambda\vartheta$ is not necessarily a magnetic DN map. For this to be the case, $\vartheta$ must be the restriction of a function that lives on the whole manifold $N$ as above. 
\end{rem}
This identity will be useful for initiating the induction procedure described in Section \ref{Section 5.3}. 

\subsection{Recovery of the magnetic potential at the boundary}

In this part, we deal with the information provided by the subprincipal order of $Q$. This will give \textbf{Theorem \ref{Th 1.2}} and by the same a first step in the direction of the boundary determination.
The following lemma shows that the relevant quantity appearing at the subprincipal order is naturally the magnetic potential :

\begin{lem}\label{Lemma 5.4}
Let $a$ and $\tilde{a}$ be two (normalized) magnetic potentials on $N$. For any (normalized) electric potentials $q$ and $\tilde{q}$ on $N$, the following identity holds in the unit cotangent bundle $S^{*}M$ :
        $$ {\rm sub}(\Lambda_{\tilde{a},\tilde{q}})-{\rm sub}(\Lambda_{a,q})=\pi_{1}^{*}\left(\tilde{a}-a\right)\hspace{0.1cm}.$$ 
\end{lem}
\begin{proof}
Let us denote in boundary normal coordinates $(x,\xi)$ 
$$\sigma^{\rm full}_{\Lambda_{\tilde{a},\tilde{q}}}\sim \sum_{j=-1}^{+\infty}\tilde{p}_{-j} \quad 
\text{ and } \quad \sigma^{\rm full}_{\Lambda_{a,q}}\sim \sum_{j=-1}^{+\infty}p_{-j}\hspace{0.2cm}.$$
Then using the expression of $\tilde{p}_{1}$ and $\tilde{p}_{0}$ provided by \cite[Lemma 8.7]{DSFKSU}, and denoting $a=\sum_{i=1}^{n}a_{\alpha} \dd x_{\alpha}$ and $\tilde{a}=\sum_{i=1}^{n}\tilde{a}_{\alpha} \dd x_{\alpha}$, we get the following identity : 
\begin{align*} 
{\rm sub}(\Lambda_{\tilde{a},\tilde{q}}) &=\tilde{p}_{0}-\frac{1}{2i}\sum_{j=1}^{n}\frac{\partial^{2}\tilde{p}_{1}}{\partial x_{j}\partial\xi_{j}}
= {\rm sub}(\Lambda_{a,q}) +\tilde{p}_{0}-p_{0}\\
&= {\rm sub}(\Lambda_{a,q}) \hspace{0.2cm}+\sum_{1\leqslant\alpha,\beta\leqslant n-1}g^{\alpha\beta}\left(\tilde{a}_{\alpha}-a_{\alpha}\right)\xi_{\beta}\\
&= {\rm sub}(\Lambda_{a,q})+\left(\tilde{a}-a\right)\xi^{\sharp}
\end{align*}
on the unit cotangent bundle $S^{*}M$.
\end{proof}
Exploiting the information provided by the wave invariants, we are now able to prove \textbf{Theorem \ref{Th 1.2}} in the same way as \textbf{Theorem \ref{Th 1.1}} for the Laplacian case :

\begin{proof}[Proof of \textbf{Theorem \ref{Th 1.2}}]
Since $\Lambda_{a,q}$ and $\Lambda_{\tilde{a},\tilde{q}}$ have the same spectrum, the principal wave invariants give that for any periodic geodesic $\gamma$ of $M$ :
\[ \int_{\gamma} \pi_{1}^{*}\left(\tilde{a}-a\right) \in 2\pi\Z \hspace{0.1cm}.\] 
Then by \textbf{Proposition \ref{Th 3.1}}, there exists a solution  $u \in C^{\infty}(S^*M;\mathbb{S}^{1})$ of the transport equation with potential $f:=\pi_{1}^{*}\left(\tilde{a}-a\right)$ :
     \[ Hu+ifu = 0 \hspace{0.1cm}.\]
Now using \textbf{Theorem \ref{RigidityTransportEq}}, we get that the one form $\restriction{\left(\tilde{a}-a\right)}{M}$ is closed, $\vartheta:=u$ does not depend on the fiber variable, and by \textbf{Lemma \ref{Lemma 4.2}}, satisfies the desired identity : 
\[ \restriction{\tilde{a}}{M} = \restriction{a}{M}-i\bar{\vartheta}\dd \vartheta \hspace{0.1cm}.\] 

\medskip\noindent Regarding the recovery of electric potential at the boundary, we refer to \textbf{Lemma \ref{Lemma 5.7}} in the next section
\end{proof}

\begin{rem}
Note that unlike the case of the Laplacian, the subprincipal order does not give information on the electric potential at order $0$. As we shall see in the following, this information is in fact contained in the homogeneous part of order $-1$.  
\end{rem}

\subsection{Determination of the jet at the boundary}\label{Section 5.3}
In this part, we proceed to the recovery of the full jet at the boundary of the magnetic field and of the electric potential, from the magnetic Steklov spectrum. First, we want to precisely describe the information contained in the homogeneous terms of order $-j\leq -1$ in the full symbol of $Q:=\Lambda_{a,q}-\Lambda_{\widetilde{a},\widetilde{q}}$ :

\begin{lem}\label{Lemma 5.6}
Let $a$ and $\tilde{a}$ be two (normalized) magnetic potentials on $N$ and $q$ and $\widetilde{q}$ be two (normalized) electric potentials on $N$.
Let $$\sigma^{\rm full}_{\Lambda_{\tilde{a},\tilde{q}}}\sim \sum_{j=-1}^{+\infty}\tilde{p}_{-j} \quad
\text{ and } \quad \sigma^{\rm full}_{\Lambda_{a,q}}\sim \sum_{j=-1}^{+\infty}p_{-j}$$ 

\medskip\noindent be respectively the full symbols of $\Lambda_{\tilde{a},\tilde{q}}$ and $\Lambda_{a,q}$ considered in any local coordinate system 
$x':=(x_{1},...x_{n-1})$ on the boundary. 

\medskip\noindent In boundary normal coordinates $(x',  x_{n})$, denoting $a=\sum_{\alpha=1}^{n}a_{\alpha}dx_{\alpha}$ and $\tilde{a}=\sum_{\alpha=1}^{n}\tilde{a}_{\alpha}dx_{\alpha}$, the following identity holds on the unit cotangent bundle $S^{*}M$, for any $j\geq 1$ :
$$p_{-j}-\tilde{p}_{-j}=-2^{-j}\left(\pi_{0}^{*}\restriction{\partial_{n}^{j-1}\left(q-\tilde{q}\right)}{x_{n}= 0}+\pi_{1}^{*}\restriction{\partial_{n}^{j}\left(a-\tilde{a}\right)}{x_{n}=0}\right)+ T_{j}\hspace{0.1cm},$$    

\medskip\noindent where 
$$T_{j}:=T_{j}\left(g^{\alpha\beta},\left(a-\tilde{a}\right)_{\alpha},\left(q-\tilde{q}\right)\right)$$

\medskip\noindent is an expression involving only the boundary values of $g^{\alpha\beta}$, $\left(a-\tilde{a}\right)_{\alpha}$, $q-\tilde{q}$ (for $j\geq 2$) and their normal derivatives respectively of order at most $j$, $j-1$ and $j-2$ at the boundary.
\end{lem}
\begin{proof}
This result is a direct consequence of the expression in boundary normal coordinates of the full symbol of the magnetic DN map (cf. \cite[Lemma 8.7]{DSFKSU} or \cite[Theorem 3.4]{Cek}).  
\end{proof}

\medskip\noindent As mentioned above (cf. Subsection \ref{Section 1.1}), the relevant quantities that appear at order $-j\leq -1$ involve the normal derivative at the boundary of both magnetic and electric potentials, respectively, at order $j$ and $j-1$. However, unlike the problems considered in \cite{Flo}, here we have to deal with the reminder terms $T_{j}$, $-j\leq -1$ at each step of induction since we are only able to recover the magnetic potential at the boundary up to gauge, i.e. up to a term of the form $-i\bar{\vartheta}\dd\vartheta$ for some $\vartheta\in C^{\infty}(M;\mathbb{S}^{1})$. \\First, using \textbf{Theorem \ref{Th 1.2}}, we know that if two magnetic potentials are isospectral, there exists $\vartheta_{0}:=\vartheta\in C^{\infty}(M;\mathbb{S}^{1})$ s.t 

\begin{align}\restriction{\tilde{a}}{M} = \restriction{a}{M}-i\bar{\vartheta_{0}}\dd \vartheta_{0} \hspace{0.1cm}.\label{subprincipal}\end{align}

\medskip\noindent Next, our goal is to obtain the same type of identity for all normal derivatives at the boundary. To describe the idea, we begin with the order $1$, this is the object of \textbf{Lemma \ref{Lemma 5.7}} below. The main ingredient is the trace formula of \textbf{Theorem \ref{Th 2.3}} but we cannot apply it directly to the pseudodifferential operator $Q$, since it is not necessarily in $\Psi^{-1}(M)$, because of the gauge term $-i\bar{\vartheta_{0}}\dd \vartheta_{0}$ in \eqref{subprincipal}. To deal with this, we need to conjugate the operator $\Lambda_{a,q}$ with $\vartheta_{0}$, that is, consider the operator 

$$Q_{1}:=\bar{\vartheta}_{0}\Lambda_{a,q}\vartheta_{0}-\Lambda_{\tilde{a},\tilde{q}}\in\Psi^{-1}(M)\hspace{0.1cm}.$$

\medskip\noindent The fact that $Q_{1}$ is in $\Psi^{-1}(M)$ is a direct consequence of how the subprincipal symbol changes after conjugacy (see \textbf{Lemma \ref{Lemma 5.2}}). The other thing to verify is that the principal symbol of $Q_{1}$ also gives the desired information on the potentials. In other words, we want to be able to apply \textbf{Lemma \ref{Lemma 5.6}} to $Q_{1}$, which requires that $\bar{\vartheta}_{0}\Lambda_{a,q}\vartheta_{0}$ be a magnetic map. This is the case if $\vartheta_{0}$ is a true gauge which is not guaranteed unless $\vartheta_{0}$ can be extended in a suitable way over the entire manifold $N$ as mentioned in \textbf{Remark \ref{Rem 5.3}}. It is of course not possible in general to extend it to a function with values in $\mathbb{S}^{1}$ due to standard topological obstructions.
The approach we use to ensure that we are working with a magnetic DN map is to consider the restrictions of $\bar{\vartheta}_{0}\Lambda_{a,q}\vartheta_{0}$ to each open chart $U_{k}$ of a finite cover $\left(U_{k}\right)_{k=1,...,K}$ of $M$ such that $$\restriction{\vartheta_{0}}{U_{k}}=\e^{i\psi_{k}}\hspace{0.1cm},\hspace{0.1cm}\psi_{k}\in C^{\infty}(U_{k},\mathbb{R})\hspace{0.1cm}.$$ 
Such a cover always exists since $t \mapsto \e^{it}$ is a covering map of $\mathbb{S}^{1}$ and $M$ is compact. Let us also note $\left(\chi_{k}\right)_{k=1,...K}$ a partition of unity subordinate to the cover. 
Then for any $k=1,...,K$, we note $\overline{W}_{k}:=Supp\left(\chi_{k}\right)\subset U_{k}$ and  $\tilde{\psi}_{k}$
a smooth extension of $\restriction{\psi_{k}}{\overline{W}_{k}}$ to $N$. As we will see below, for our proof to work, we must choose this extension so that it is flat at the boundary in the normal direction, in summary we choose $\tilde{\psi}_{k}$ such that for all $k=1,...,K$, 
\begin{itemize}
\item[$\bullet$] $\tilde{\psi}_{k}$ is a smooth map on $N$ and $$\restriction{\tilde{\psi}_{k}}{\overline{W}_{k}}=\psi_{k}\hspace{0.1cm}.$$
\item[$\bullet$] All its normal derivatives at the boundary vanish : $$\forall \ell\geq 1\hspace{0.1cm},\hspace{0.1cm}\restriction{\partial_{n}^{\ell}\tilde{\psi}_{k}}{x_{n}=0}=0$$ in any boundary coordinates $(x',x_{n})$.
\end{itemize}
It is standard to see that this choice is always possible. We are now able to recover the electric potential at the boundary and, up to gauge, the normal derivative of the magnetic potential at the boundary from the magnetic Steklov spectrum :
     
\begin{lem}\label{Lemma 5.7}
Assume that the cogeodesic flow $\phi_t$ of $M$ is Anosov with simple length spectrum. If the magnetic DN maps $\Lambda_{a,q}$ and $\Lambda_{\tilde{a},\tilde{q}}$ have the same eigenvalues with multiplicities
 \[ \big(\sigma_k(\Lambda)\big)_{k \in \N} = \big(\sigma_k(\tilde{\Lambda})\big)_{k \in \N} \hspace{0.1cm},\]
then 
$$\restriction{\tilde{q}}{M}=\restriction{q}{M}$$

\medskip\noindent and the $1$-form $\restriction{\partial_{\nu}\left(\tilde{a}-a\right)}{M}$ is exact, i.e. there is $\beta_{1}\in C^{\infty}(M)$ such that :

$$\restriction{\partial_{\nu}\tilde{a}}{M}=\restriction{\partial_{\nu}a}{M} +\dd\beta_{1}\hspace{0.1cm}.$$

\medskip\noindent In particular, there exists $\vartheta_{1}\in C^{\infty}(M;\mathbb{S}^{1})$ s.t

$$\restriction{\partial_{\nu}\tilde{a}}{M}=\restriction{\partial_{\nu}a}{M}-i\bar{\vartheta_{1}}\dd\vartheta_{1}\hspace{0.1cm}.$$ 
\end{lem}
\begin{proof}
Since $\Lambda_{\tilde{a},\tilde{q}}$ and $\bar{\vartheta}_{0}\Lambda_{a,q}\vartheta_{0}$ have the same spectrum and $M$ is Anosov with simple length spectrum, one can apply \textbf{Corollary \ref{Cor 2.4}} to the operators $P:=\Lambda_{\tilde{a},\tilde{q}}\in\Psi^{1}(M)$ and $Q_{1}\in\Psi^{-1}(M)$. We get that for any closed orbit $\gamma$ of the geodesic flow on $M$ :

    $$\int_{\gamma}\sigma_{Q_{1}}=0\hspace{0.1cm}.$$

\medskip\noindent Given the expression of \textbf{Lemma \ref{Lemma 5.6}}, we observe that on the unit cotangent bundle $S^{*}M$, the principal symbol of $Q_{1}\in\Psi^{-1}(M)$ is the sum of a function $f_{0}\in C^{\infty}(M)$ and a $1$-form $f_{1}\in\Omega_{1}(M)$ : $$\sigma_{Q_{1}}=\pi_{0}^{*}f_{0}+\pi_{1}^{*}f_{1}$$ Then, once again by injectivity of the X-ray transform on functions and $1$-forms (see \cite{DS}), we obtain that $f_{0}=0$ and $f_{1}$ is exact $$f_{1}:=\dd\beta_{1}\hspace{0.1cm}.$$
In addition, since for any $k\in\lbrace 1,...,K\rbrace$, the principal symbols of $\bar{\vartheta}_{0}\Lambda_{a,q}\vartheta_{0}$ and $\Lambda_{a+\dd\tilde{\psi}_{k},q}$ coincide on $S^{*}\overline{W}_{k}$ (it is easy to see by using for example the partition of unity $(\chi_{k})_{k=1,...K}$), we get

$$\restriction{\sigma_{Q_{1}}}{S^{*}\overline{W}_{k}}=\restriction{\left(\sigma_{\Lambda_{a+\dd\tilde{\psi}_{k},q}}-\sigma_{\Lambda_{\tilde{a},\tilde{q}}}\right)}{S^{*}\overline{W}_{k}}\hspace{0.1cm}.$$

\medskip\noindent Next, using \textbf{Lemma \ref{Lemma 5.6}} we deduce that in boundary normal coordinates $(x',x_{n})$, for any $k\in\lbrace1,...,K\rbrace$, $$\restriction{\sigma_{Q_{1}}}{S^{*}\overline{W}_{k}}=-\frac{1}{2}\left(\pi_{0}^{*}\restriction{\left(q-\tilde{q}\right)}{S^{*}\overline{W}_{k}}+\pi_{1}^{*}\restriction{\partial_{n}\left(a-\tilde{a}+\dd\tilde{\psi}_{k}\right)}{S^{*}\overline{W}_{k}}\right)+ T_{1}\hspace{0.1cm},$$
where $T_{1}=0$ since it involves only the values of $g^{\alpha\beta}$, $\partial_{n}g^{\alpha\beta}$, $\left(a+\dd\tilde{\psi}_{k}-\tilde{a}\right)_{\alpha}$ in $\overline{W}_{k}\subset M$ and also by \eqref{subprincipal} :
$$\restriction{\tilde{a}}{\overline{W}_{k}} = \restriction{\left(a+\dd\tilde{\psi}_{k}\right)}{\overline{W}_{k}}\hspace{0.1cm}.$$
Then using the fact that $\sigma_{Q_{1}}$ is exact and that for all $k$, $\tilde{\psi}_{k}$ is flat at the boundary in the normal direction, we obtain the desired identities on each $\overline{W}_{k}$ and then on the whole manifold $M$ since $(\overline{W}_{k})_{k=1,...,K}$ covers $M$.
In particular $$\restriction{\partial_{n}\left(\tilde{a}-a\right)}{x_{n}=0}=\dd\beta_{1}$$ 

\medskip\noindent is closed and \textbf{Lemma \ref{Lemma 4.2}} gives $\vartheta_{1}$ as desired.
\end{proof}

\medskip\noindent Unlike the information \eqref{subprincipal} obtained on the magnetic potential at the subprincipal order, here we obtain a gauge term that is exact, which is much more convenient for continuing the induction. Consequently, one can proceed in the same way to recover the normal derivatives of higher order and it is even easier as it is possible to construct a true gauge transformation at the end of each step of the induction. More explicitly, assume that for $j\in\mathbb{N}^{*}$, we have constructed $\left(\beta_{\ell}\right)_{\ell=1,...,j}$ a family of smooth functions on $N$ such that for all $1\leq \ell\leq j$, 

\begin{align}\restriction{\partial_{n}^{\ell}\tilde{a}}{x_{n}=0}=\restriction{\partial_{n}^{\ell}a}{x_{n}=0}+\dd\beta_{\ell}(x',0)\hspace{0.1cm}.\label{normderiv}\end{align}

\medskip\noindent in boundary normal coordinates near $M$. Then one can define the smooth function $$\alpha_{j}:=\sum_{\ell=1}^{j}\chi\cdot\beta_{\ell}\frac{x_{n}^{\ell}}{\ell\mathpunct{}!}$$
on $N$, where $\chi\in C^{\infty}(N)$ obviously denotes a smooth bump function. Unlike $\vartheta_{0}$, the function $e^{i\alpha_{j}}$ defines a true gauge in the sense that a magnetic DN map $\Lambda_{b,q}$ conjugated by $e^{i\alpha_{j}}$ (as a unitary operator of order $0$) remains a magnetic DN map (see Section \ref{Subsection 5.1}) : 

$$\e^{-i\alpha_{j}}\Lambda_{b,q}e^{i\alpha_{j}}=\Lambda_{b+\dd\alpha_{j},q}\hspace{0.1cm}.$$

\medskip\noindent It is then natural to consider the pseudodifferential operator 

$$Q_{j+1}:=\e^{-i\alpha_{j}}\bar{\vartheta}_{0}\Lambda_{a,q}\vartheta_{0}\e^{i\alpha_{j}}-\Lambda_{\tilde{a},\tilde{q}}$$ 

\medskip\noindent to apply our wave trace formula (cf. \textbf{Theorem \ref{Th 2.3}}) in the same way as we apply it to $Q_{1}$ in the proof of \textbf{Lemma \ref{Lemma 5.7}}.

\begin{lem}\label{Lemma 5.8}
Assume that for all $1\leq \ell\leq j$, 

   \begin{align}\restriction{\partial_{n}^{\ell-1}\left(\tilde{q}-q\right)}{x_{n}=0}=0\hspace{0.1cm}.\label{potential}\end{align}

\medskip\noindent Then
$$Q_{j+1}\in\Psi^{-j-1}(M)\hspace{0.1cm}.$$ 
\end{lem}
\begin{proof}
Let $$\sigma^{\rm full}_{\Lambda_{\tilde{a},\tilde{q}}}\sim \sum_{j=-1}^{+\infty}\tilde{p}_{-j} \quad \text{ and }
\quad \sigma^{\rm full}_{B_{j}}\sim \sum_{j=-1}^{+\infty}b_{-j}$$

\medskip\noindent be respectively the full symbols of $\Lambda_{\tilde{a},\tilde{q}}$ and of 

$$B_{j}:=\e^{-i\alpha_{j}}\bar{\vartheta}_{0}\Lambda_{a,q}\vartheta_{0}\e^{i\alpha_{j}}$$ 

\medskip\noindent considered in any local coordinate system 
$x':=(x_{1},...x_{n-1})$ on the boundary. 
The key point here, and once again the principle of our idea, is to observe that the full symbol of $B_{j}$ coincide on each $S^{*}\overline{W}_{k}$ with the full symbol of $\Lambda_{a+\dd(\tilde{\psi_{k}}+\alpha_{j})}$ in any boundary coordinates since locally, 
$$B_{j}=\Lambda_{a+\dd(\tilde{\psi_{k}}+\alpha_{j})}$$
as operators in $U_{k}$.
Then using \textbf{Lemma \ref{Lemma 5.6}}, we get, for any $k\in\lbrace1,...,K\rbrace$ and any $1\leq \ell\leq j$, the following identity on $S^{*}\overline{W}_{k}$, in boundary normal coordinates $(x',x_{n})$ near $\overline{W}_{k}$ : $$\begin{aligned}[t] b_{-\ell}-\tilde{p}_{-\ell}&=-2^{-\ell}\left(\pi_{0}^{*}\partial^{\ell-1}_{n}\left(q-\tilde{q}\right)+\pi_{1}^{*}\partial^{\ell}_{n}\left(a-\tilde{a}+\dd(\tilde{\psi_{k}}+\alpha_{j})\right)\right)+ T_{\ell}\\
&= -2^{-\ell}\pi_{1}^{*}\partial^{\ell}_{n}\left(a-\tilde{a}+\dd(\tilde{\psi_{k}}+\alpha_{j})\right)+ T_{\ell}\\
\end{aligned}$$
where the second equality comes from the use of \eqref{potential}. 
Note that $T_{\ell}$ involves only the values of $g^{\alpha\beta}$, $\left(a-\tilde{a}+\dd(\tilde{\psi}_{k}+\alpha_{j})\right)_{\alpha}$, $q-\tilde{q}$ and of their normal derivatives respectively of order at most $\ell$, $\ell-1$ and $\ell-2$ in $\overline{W}_{k}\subset M$. We know that the terms involving  metrics are known and those involving electric potentials are zero (by \eqref{potential}). Regarding the magnetic terms, since $\tilde{\psi}_{k}$ and $\chi$ are flat at the boundary in the normal direction (recall that $\chi$ is a bump function, so it is constant equal to $1$ in a neighborhood of $M$) i.e.
$$\forall \ell\geq 1\hspace{0.1cm},\hspace{0.1cm}\restriction{\partial_{n}^{\ell}\tilde{\psi}_{k}}{x_{n}=0}=0=\restriction{\partial_{n}^{\ell}\chi}{x_{n}=0}\hspace{0.1cm},$$
then $T_{\ell}$ actually depends only on the values in $\overline{W}_{k}$ of $\left(a-\tilde{a}+\dd\tilde{\psi}_{k}\right)_{\alpha}$ and of $\partial_{n}\left(a-\tilde{a}+\dd\beta_{1}\right)_{\alpha},...,\partial^{j}_{n}\left(a-\tilde{a}+\dd\beta_{j}\right)_{\alpha}\cdot$ \\By \eqref{subprincipal} and \eqref{normderiv}, these values are zero and we finally conclude that $T_{\ell}=0$. In the same way, using \eqref{subprincipal}, \eqref{normderiv} and the fact that $\tilde{\psi}_{k}$ and $\chi$ are flat at the boundary in the normal direction, we see that for any $1\leq\ell\leq j$, $$b_{-\ell}-\tilde{p}_{-\ell}=-2^{-\ell}\left(\pi_{1}^{*}\partial^{\ell}_{n}\left(a-\tilde{a}+\dd(\tilde{\psi_{k}}+\alpha_{j})\right)\right)=0$$ on each $S^{*}\overline{W}_{k}$ and then in $S^{*}M$ since $(\overline{W}_{k})_{k=1,...,K}$ cover $M$.
\end{proof}

\medskip\noindent We can then show the following theorem which proves \textbf{Theorem \ref{Th 1.3}} :

\begin{thm}\label{Th 5.9}
Assume that the cogeodesic flow $\phi_t$ of $M$ is Anosov with simple length spectrum.  If the spectra of the magnetic DN maps $\Lambda_{a,q}$ and $\Lambda_{\tilde{a},\tilde{q}}$ coincide 
 \[ \big(\sigma_k(\Lambda)\big)_{k \in \N} = \big(\sigma_k(\tilde{\Lambda})\big)_{k \in \N} \hspace{0.1cm},\]   
     
\medskip\noindent then in boundary normal coordinates $(x',x_{n})$, for any $j\in\mathbb{N}^{*}$,
$$\restriction{\partial_{n}^{j-1}\left(\tilde{q}-q\right)}{x_{n}=0}=0$$

\medskip\noindent and there exists $\beta_{j}\in C^{\infty}(N)$ s.t 
$$\restriction{\partial_{n}^{j}\tilde{a}}{x_{n}=0}=\restriction{\partial_{n}^{j}a}{x_{n}=0}+\dd\beta_{j}(x',0)\hspace{0.1cm}.$$

\medskip\noindent In particular, for any $j\in\mathbb{N}^{*}$, there exists $\vartheta_{j}\in C^{\infty}(M;\mathbb{S}^{1})$ s.t

$$\restriction{\partial_{n}^{j}\tilde{a}}{x_{n}=0}=\restriction{\partial_{n}^{j}a}{x_{n}=0}-i\bar{\vartheta_{j}}\dd\vartheta_{j}\hspace{0.1cm}.$$   
\end{thm}
\begin{proof}
Consider $x':=(x_{1},...x_{n-1})$ any local coordinates on the boundary and $(x',  x_{n})$ the associated $g$-boundary normal coordinates near $\partial M$.
\\We proceed by induction on the order $j\geq 1$ : 
\vspace{0.3cm}
\begin{itemize}
    \item[$\bullet$] $j=1$ : This is established by \textbf{Lemma \ref{Lemma 5.7}} above.
    
    \vspace{0.2cm}
    
    \item[$\bullet$] Let $j\geq 1$ s.t for all $1\leq \ell\leq j$, 
    
    $$\restriction{\partial_{n}^{\ell-1}\left(\tilde{q}-q\right)}{x_{n}=0}=0$$

\medskip\noindent and there exists $\beta_{\ell}\in C^{\infty}(N)$ s.t 

$$\restriction{\partial_{n}^{\ell}\tilde{a}}{x_{n}=0}=\restriction{\partial_{n}^{\ell}a}{x_{n}=0}+\dd\beta_{\ell}(x',0)\hspace{0.1cm}.$$
    
\medskip\noindent Since $\Lambda_{\tilde{a},\tilde{q}}$ and $e^{-i\alpha_{j}}\bar{\vartheta}_{0}\Lambda_{a,q}\vartheta_{0}e^{i\alpha_{j}}$ have the same spectrum and $M$ is Anosov with simple length spectrum, one can apply \textbf{Corollary \ref{Cor 2.4}} to the operators $P:=\Lambda_{\tilde{a},\tilde{q}}\in\Psi^{1}(M)$ and $Q_{j+1}\in\Psi^{-j-1}(M)$. We get that for any closed orbit $\gamma$ of the geodesic flow on $M$ :

    $$\int_{\gamma}\sigma_{Q_{j+1}}=0\hspace{0.1cm}.$$

\medskip\noindent Then it is the exact same proof as that of \textbf{Lemma \ref{Lemma 5.7}}, just replacing of course $Q_{1}$ by $Q_{j+1}$, noting that for any $k\in\lbrace 1,...,K\rbrace$, $$\restriction{\sigma_{P+Q_{j+1}}}{S^{*}\overline{W}_{k}}=\restriction{\sigma_{\Lambda_{a+\dd(\tilde{\psi}_{k}+\alpha_{j}),q}}}{S^{*}\overline{W}_{k}}\hspace{0.1cm},$$ and finally combining this with \textbf{Lemma \ref{Lemma 5.6}} to deduce that for any $k\in\lbrace1,...,K\rbrace$, in boundary normal coordinates $(x',x_{n})$ near $U_{k}$, the principal symbol of $Q_{j+1}$ in $S^{*}\overline{W}_{k}$ is essentially given by the sum of $\pi_{0}^{*}\restriction{\partial_{n}^{j}\left(q-\tilde{q}\right)}{S^{*}\overline{W}_{k}}$, $\pi_{1}^{*}\restriction{\partial_{n}^{j+1}\left(a-\tilde{a}+\dd(\tilde{\psi}_{k}+\alpha_{j})\right)}{S^{*}\overline{W}_{k}}$ and of $T_{j+1}$. Moreover, $T_{j+1}=0$ since it involves only the values of $g^{\alpha\beta}$, $q-\tilde{q}$, $\left(a-\tilde{a}+\dd(\tilde{\psi}_{k}+\alpha_{j})\right)_{\alpha}$ and of their normal derivatives respectively of order at most $j+1$, $j-1$ and $j$ in $\overline{W}_{k}\subset M$. Note that we also used crucially, as well as to conclude the proof, the identities \eqref{subprincipal}, \eqref{normderiv}, \eqref{potential} and the fact that $\tilde{\psi}_{k}$ and $\chi$ are flat at the boundary in the normal direction, namely for all $\ell\geq 1$ : 
$$\restriction{\partial_{n}^{\ell}\tilde{\psi}_{k}}{x_{n}=0}=0=\restriction{\partial_{n}^{\ell}\chi}{x_{n}=0}\hspace{0.1cm}.$$
We refer to the end of the proof of \textbf{Lemma \ref{Lemma 5.8}} for a similar and more detailed use of these identities.
\end{itemize}
\end{proof}
Since the gauge on the magnetic potential is closed, it gives the recovery of the magnetic field at the boundary as expected (cf. \textbf{Theorem \ref{Th 1.3}}).

\section{Towards the framework of connections}
\label{Sec:Connexion}
Inverse problems for magnetic Schr\"odinger operators and for
magnetic DN maps can be naturally understood as
special instances of a more general geometric framework. Regarding first a magnetic Schr\"odinger operator, consider
a Hermitian vector bundle $E\to M$ over a closed Riemannian manifold $M$, a unitary
connection $\nabla$ on $E$, and a Hermitian
matrix-valued potential $Q\in C^{\infty}(M,\mathrm{End}(E))$ which is assumed to be positive (pointwise as Hermitian endomorphism). The corresponding connection Laplacian acts on sections $u\in C^{\infty}(M,E)$ by
$$
    \mathcal{L}_{\nabla,Q} := \nabla^*\nabla + Q \hspace{0.1cm}.
$$
It is a formally self-adjoint differential operator on $L^{2}(M,E)$ with discrete non-negative spectrum : $$0\leq\lambda_{1}(\nabla)\leq\lambda_{2}(\nabla)\leq...\rightarrow \infty\hspace{0.1cm}.$$ 
The spectral inverse problem consists of determining the pair
$(\nabla,Q)$ from the spectrum of $\mathcal{L}_{\nabla,Q}$. 
The usual magnetic problem is recovered as the
\emph{rank-one} case : a real $1$-form $a$ defines a unitary connection
$\nabla^{a}:=\dd+ia$ on the trivial line bundle (and of course $Q$ is a scalar electric potential $q$ in that case). As already mentioned, in absence of electric potential i.e for $q=0$, Ceki{\'c} and Lefeuvre answered the question positively when $M$ is negatively curved (cf. \cite{CekLef}) by using the wave trace formula of \cite{Guillemin}. 

In the same way, one can consider the analogue for the DN map associated with connection Laplacians :  Let $E\to N$ be a Hermitian vector bundle over a compact Riemannian manifold $N$ with boundary, equipped with a unitary connection $\nabla^A$ and a smooth matrix-valued potential
$Q\in C^\infty(N,\mathrm{End}(E))$.  
The associated connection Laplacian $\Delta_{A,Q}:=\mathcal{L}_{\nabla^A,Q}$ induces a Dirichlet-to-Neumann map 
$$
\Lambda_{A,Q} \hspace{0.05cm}f = \restriction{\nabla^A_\nu u}{\partial N}\quad,
\quad 
\restriction{u}{\partial N}=f\quad,
\quad 
\Delta_{A,Q} \hspace{0.05cm}u = 0\hspace{0.1cm},
$$
provided that $0$ is not a Dirichlet eigenvalue of $\Delta_{A,Q}$.
It is an elliptic pseudodifferential operator of order $1$ acting on sections
$f\in C^\infty(\partial N,\restriction{E}{\partial N})$. Moreover, it is not difficult to see that $\Lambda_{A,Q}$ is self-adjoint if $Q$ is Hermitian matrix-valued. Thus, in this case, as usual, it admits a discrete spectrum of eigenvalues : $$\sigma_{1}(\nabla^A)\leq\sigma_{2}(\nabla^A)\leq...\rightarrow \infty\hspace{0.1cm}.$$
What we could call the \emph{connection Steklov inverse problem} consists in determining the gauge equivalence class of $(A,Q)$ from the sequence of eigenvalues of $\Lambda_{A,Q}$. Note that gauge transformations act by
$$
A \mapsto U^{-1} A U + U^{-1} \dd U\quad,
\quad 
Q \mapsto U^{-1} Q U\quad,
\quad 
U\in C^\infty\left(N,\mathbb{U}(r)\right)\hspace{0.1cm}.
$$
If $E$ is the trivial complex line bundle, $$
A:=a\in\Omega^{1}(N;\R)\quad,
\quad 
Q := q\in C^{\infty}(N)\quad,
\quad
\nabla^a := \dd + ia\hspace{0.1cm}, 
$$
and the operator $\Delta_{A,Q}$ becomes the magnetic Schr\"odinger operator
$$
\Delta_{a,q} = (\dd+ia)^{*}(\dd+ia)+q\hspace{0.1cm}.
$$
Thus the recovery of $(a,q)$ from the magnetic Steklov spectrum is simply the 
rank-one case of the spectral inverse problem for $(A,Q)$.

The DN map $\Lambda_{A,Q}$ is a classical matrix-valued pseudodifferential operator
with full symbol
$$
\sigma^{\rm full}_{\Lambda_{A,Q}}(x,\xi)
\sim
|\xi|\,\mathrm{Id}
\,+\,\sigma_0(A,Q)(x,\xi)
\,+\,\sigma_{-1}(A,Q)(x,\xi)
\,+\cdots\hspace{0.1cm}.
$$
While the principal symbol is a scalar multiple of the identity and depends only on the metric, the lower-order homogeneous terms $\sigma_{-j}(A,Q)$, $j\geq 0$, encodes the full Taylor series of $A$ and $Q$ (see \cite[Theorem 3.4]{Cek}).

We believe that the same approach used in this article can work to relate these symbols to the Steklov spectrum, namely by using wave trace techniques in the setting of pseudodifferential operators on vector bundles. For example, the information contained at subprincipal order could be recovered from the leading terms in the wave trace formula of \cite{Guillemin} which essentially consist (under the assumptions that $\partial N$ is Anosov with simple length spectrum) of the traces of parallel transports $\mathrm{P}_{\gamma}$ along closed geodesics $\gamma$ of period $T_{\gamma}$ :
$$\left(t-T_{\gamma}\right)\sum_{j=0}^{+\infty}e^{it\sigma_{j}(\nabla^A)}\underset{t\rightarrow T_{\gamma}}{\longrightarrow}\mathrm{Tr}\left(\mathrm{P}_{\gamma}\right)\frac{\lvert T_{\gamma}^{\sharp}\rvert}{2\pi}\lvert \det\left({\rm Id}-\mathcal{P}_{j}\right)\rvert^{-1/2}\hspace{0.1cm}.$$
The natural continuation would be to ask whether it is possible to find an appropriate version of this trace formula in the manner of \textbf{Theorem \ref{Th 2.3}} to deal with lower-order terms and, hopefully, obtain results analogous to \textbf{Theorem \ref{Th 1.3}} in the context of connections and matrix-valued potentials.
We hope to address these topics in future work.

\bibliographystyle{alphaurl}
\bibliography{biblio}

@Misc{Ab,
 Author = {Abraham, Ralph},
 Title = {Bumpy metrics},
 Year = {1970},
 Language = {English},
 HowPublished = {Global {Analysis}, {Proc}. {Sympos}. {Pure} {Math}. 14, 1-3 (1970).},
 zbMATH = {3341663},
 Zbl = {0215.23301}
}

@article{AhBo,
 author = {Aharonov, Y. and Bohm, D.},
 title = {Significance of electromagnetic potentials in the quantum theory},
 fjournal = {Physical Review, II. Series},
 journal = {Phys. Rev., II. Ser.},
 issn = {0031-899X},
 volume = {115},
 pages = {485--491},
 year = {1959},
 language = {English},
 doi = {10.1103/PhysRev.115.485},
 zbMATH = {3162712},
 Zbl = {0099.43102}
}

@Article{An,
 Author = {Anosov, Dmitriĭ Viktorovich},
 Title = {On generic properties of closed geodesics},
 FJournal = {Mathematics of the USSR. Izvestiya},
 Journal = {Math. USSR, Izv.},
 ISSN = {0025-5726},
 Volume = {21},
 Pages = {1--29},
 Year = {1983},
 Language = {English},
 DOI = {10.1070/IM1983v021n01ABEH001637},
 zbMATH = {3883268},
 Zbl = {0554.58043}
}

@article{BCDSFKS,
 author = {Bellassoued, Mourad and Choulli, Mourad and Dos Santos Ferreira, David and Kian, Yavar and Stefanov, Plamen},
 title = {A {Borg}-{Levinson} theorem for magnetic {Schr{\"o}dinger} operators on a {Riemannian} manifold},
 fjournal = {Annales de l'Institut Fourier},
 journal = {Ann. Inst. Fourier},
 issn = {0373-0956},
 volume = {71},
 number = {6},
 pages = {2471--2517},
 year = {2021},
 language = {English},
 doi = {10.5802/aif.3451},
 zbMATH = {7554452},
 Zbl = {1517.35255}
}

@Article{Br,
 Author = {Brooks, Robert},
 Title = {On manifolds of negative curvature with isospectral potentials},
 FJournal = {Topology},
 Journal = {Topology},
 ISSN = {0040-9383},
 Volume = {26},
 Pages = {63--66},
 Year = {1987},
 Language = {English},
 DOI = {10.1016/0040-9383(87)90021-8},
 zbMATH = {4000926},
 Zbl = {0617.53048}
}

@Article{BrPeYa,
 Author = {Brooks, Robert and Perry, Peter and Yang, Paul},
 Title = {Isospectral sets of conformally equivalent metrics},
 FJournal = {Duke Mathematical Journal},
 Journal = {Duke Math. J.},
 ISSN = {0012-7094},
 Volume = {58},
 Number = {1},
 Pages = {131--150},
 Year = {1989},
 Language = {English},
 DOI = {10.1215/S0012-7094-89-05808-0},
 zbMATH = {4092284},
 Zbl = {0667.53037}
}

@Article{Cek,
 Author = {Ceki{\'c}, Mihajlo},
 Title = {Calder{\'o}n problem for {Yang}-{Mills} connections},
 FJournal = {Journal of Spectral Theory},
 Journal = {J. Spectr. Theory},
 ISSN = {1664-039X},
 Volume = {10},
 Number = {2},
 Pages = {463--513},
 Year = {2020},
 Language = {English},
 DOI = {10.4171/JST/302},
 zbMATH = {7235850},
 Zbl = {1445.35326}
}

@article{CekLef,
 author = {Ceki{\'c}, Mihajlo and Lefeuvre, Thibault},
 title = {Isospectral connections, ergodicity of frame flows, and polynomial maps between spheres},
 fjournal = {Annales Scientifiques de l'{\'E}cole Normale Sup{\'e}rieure. Quatri{\`e}me S{\'e}rie},
 journal = {Ann. Sci. {\'E}c. Norm. Sup{\'e}r. (4)},
 issn = {0012-9593},
 volume = {58},
 number = {1},
 pages = {203--229},
 year = {2025},
 language = {English},
 doi = {10.24033/asens.2604},
 zbMATH = {8048370}
}

@article{CekSif,
 author = {Ceki{\'c}, Mihajlo and Siffert, Anna},
 title = {Magnetic {Steklov} problem on surfaces},
 fjournal = {Journal of Functional Analysis},
 journal = {J. Funct. Anal.},
 issn = {0022-1236},
 volume = {289},
 number = {12},
 pages = {47},
 note = {Id/No 111159},
 year = {2025},
 language = {English},
 doi = {10.1016/j.jfa.2025.111159},
 zbMATH = {8092609}
}

@misc{CGHP,
 author = {Chakradhar, Tirumala and Gittins, Katie and Habib, Georges and Peyerimhoff, Norbert},
 title = {A note on the magnetic {Steklov} operator on functions},
 year = {2024},
 howpublished = {Preprint, {arXiv}:2410.07462 [math.{DG}] (2024)},
 url = {https://arxiv.org/abs/2410.07462},
 arXiv = {arXiv:2410.07462}
}

@article{DS,
 author = {Dairbekov, Nurlan S. and Sharafutdinov, Vladimir A.},
 title = {Some problems of integral geometry on {Anosov} manifolds.},
 fjournal = {Ergodic Theory and Dynamical Systems},
 journal = {Ergodic Theory Dyn. Syst.},
 issn = {0143-3857},
 volume = {23},
 number = {1},
 pages = {59--74},
 year = {2003},
 language = {English},
 doi = {10.1017/S0143385702000822},
 zbMATH = {2015884},
 Zbl = {1140.58302}
}

@incollection{DatHez,
 author = {Datchev, Kiril and Hezari, Hamid},
 title = {Inverse problems in spectral geometry},
 booktitle = {Inverse problems and applications. Inside out II},
 isbn = {978-1-107-03201-9},
 pages = {455--485},
 year = {2013},
 publisher = {Cambridge: Cambridge University Press},
 language = {English},
 zbMATH = {6462208},
 Zbl = {1316.35001}
}

@Article{DG,
 Author = {Duistermaat, Johannes Jisse and Guillemin, Victor W.},
 Title = {The spectrum of positive elliptic operators and periodic bicharacteristics},
 FJournal = {Inventiones Mathematicae},
 Journal = {Invent. Math.},
 Volume = {29},
 Pages = {39--79},
 Year = {1975},
 DOI = {10.1007/BF01405172},
 zbMATH = {3481135},
 Zbl = {0307.35071}
}

@Article{DH,
 Author = {Duistermaat, Johannes Jisse and H{\"o}rmander, Lars},
 Title = {Fourier integral operators. {II}},
 FJournal = {Acta Mathematica},
 Journal = {Acta Math.},
 Volume = {128},
 Pages = {183--269},
 Year = {1972},
 DOI = {10.1007/BF02392165},
 zbMATH = {3367678},
 Zbl = {0232.47055}
}

@article{DSFKSU,
    author = {Dos Santos Ferreira, David and Kenig, Carlos E. and Salo, Mikko and Uhlmann, Gunther},
    title = "Limiting {Carleman} weights and anisotropic inverse problems",
    journal = "Invent. Math.",
    volume = {178},
    number = {1},
    pages = {119--171},
    year = "2009",
    DOI = "10.1007/s00222-009-0196-4"
}

@article{Esk,
 author = {Eskin, Gregory},
 title = {Aharonov-Bohm effect revisited},
 fjournal = {Reviews in Mathematical Physics},
 journal = {Rev. Math. Phys.},
 issn = {0129-055X},
 volume = {27},
 number = {2},
 pages = {52},
 note = {Id/No 1530001},
 year = {2015},
 language = {English},
 doi = {10.1142/S0129055X15300010},
 zbMATH = {6445352},
 Zbl = {1342.35450}
}

@article{EskRals1,
 author = {Eskin, G. and Ralston, J.},
 title = {Gauge equivalence and the inverse spectral problem for the magnetic {Schr{\"o}dinger} operator on the torus},
 fjournal = {Russian Journal of Mathematical Physics},
 journal = {Russ. J. Math. Phys.},
 issn = {1061-9208},
 volume = {20},
 number = {4},
 pages = {413--423},
 year = {2013},
 language = {English},
 doi = {10.1134/S1061920813040043},
 zbMATH = {6435417},
 Zbl = {1311.81110}
}

@Book{FisherHass,
 Author = {Fisher, Todd and Hasselblatt, Boris},
 Title = {Hyperbolic flows},
 FSeries = {Zurich Lectures in Advanced Mathematics},
 Series = {Zur. Lect. Adv. Math.},
 ISBN = {978-3-03719-200-9; 978-3-03719-700-4},
 Year = {2019},
 Publisher = {Berlin: European Mathematical Society (EMS)},
 Language = {English},
 DOI = {10.4171/200},
 zbMATH = {7138997},
 Zbl = {1430.37002}
}

@misc{Flo,
 author = {Florentin, Benjamin},
 title = {Steklov isospectrality of conformal metrics},
 year = {2026},
 howpublished = {Preprint, {arXiv}:2501.15535 [math.{SP}] (2026)},
 url = {https://arxiv.org/abs/2501.15535},
 arXiv = {arXiv:2501.15535}
}

@Misc{Guillemin,
 Author = {Guillemin, Victor W.},
 Title = {Fourier integral operators for systems},
 Year = {1974},
 HowPublished = {Unpublished report},
 URL = {https://msp.org/extras/Guillemin/V.Guillemin-Fourier_Integral_Operators-1974.pdf},
}

@Article{Gui,
 Author = {Guillemin, Victor W.},
 Title = {Some spectral results on rank one symmetric spaces},
 FJournal = {Advances in Mathematics},
 Journal = {Adv. Math.},
 Volume = {28},
 Pages = {129--137},
 Year = {1978},
 DOI = {10.1016/0001-8708(78)90059-2},
 zbMATH = {3688358},
 Zbl = {0441.58012}
}

@article{G1,
 author = {Guillemin, Victor},
 title = {Some spectral results for the {Laplace} operator with potential on the n-sphere},
 fjournal = {Advances in Mathematics},
 journal = {Adv. Math.},
 issn = {0001-8708},
 volume = {27},
 pages = {273--286},
 year = {1978},
 language = {English},
 doi = {10.1016/0001-8708(78)90102-0},
 zbMATH = {3674602},
 Zbl = {0433.35052}
}

@article{G2,
 author = {Guillemin, V.},
 title = {An {Addendum} to: {Some} spectral results on rank one symmetric spaces},
 fjournal = {Advances in Mathematics},
 journal = {Adv. Math.},
 issn = {0001-8708},
 volume = {28},
 pages = {138--147},
 year = {1978},
 language = {English},
 doi = {10.1016/0001-8708(78)90060-9},
 zbMATH = {3688359},
 Zbl = {0441.58013}
}

@article{G3,
 author = {Guillemin, Victor},
 title = {Wave-trace invariants},
 fjournal = {Duke Mathematical Journal},
 journal = {Duke Math. J.},
 issn = {0012-7094},
 volume = {83},
 number = {2},
 pages = {287--352},
 year = {1996},
 language = {English},
 doi = {10.1215/S0012-7094-96-08311-8},
 zbMATH = {912128},
 Zbl = {0858.58051}
}

@article{G4,
 author = {Guillemin, V.},
 title = {Inverse spectral results on two-dimensional tori},
 fjournal = {Journal of the American Mathematical Society},
 journal = {J. Am. Math. Soc.},
 issn = {0894-0347},
 volume = {3},
 number = {2},
 pages = {375--387},
 year = {1990},
 language = {English},
 doi = {10.2307/1990958},
 zbMATH = {4151475},
 Zbl = {0702.58075}
}

@Article{GK1,
 Author = {Guillemin, Victor W. and Kazhdan, David},
 Title = {Some inverse spectral results for negatively curved 2-manifolds},
 FJournal = {Topology},
 Journal = {Topology},
 Volume = {19},
 Pages = {301--312},
 Year = {1980},
 DOI = {10.1016/0040-9383(80)90015-4},
 zbMATH = {3729145},
 Zbl = {0465.58027}
}

@Misc{GK2,
 Author = {Guillemin, Victor W. and Kazhdan, David},
 Title = {Some inverse spectral results for negatively curved n-manifolds},
 Year = {1980},
 Language = {English},
 HowPublished = {Geometry of the {Laplace} operator, {Honolulu}/{Hawaii} 1979, {Proc}. {Symp}. {Pure} {Math}., {Vol}. 36, 153-180 (1980).},
 zbMATH = {3714626},
 Zbl = {0456.58031}
}

@article{Helf,
 author = {Helffer, B.},
 title = {Effet d'{Aharonov}-{Bohm} sur un {\'e}tat born{\'e} de l'{\'e}quation de {Schr{\"o}dinger}. ({Aharonov}-{Bohm} effect for a bounded state of the {Schr{\"o}dinger} equation)},
 fjournal = {Communications in Mathematical Physics},
 journal = {Commun. Math. Phys.},
 issn = {0010-3616},
 volume = {119},
 number = {2},
 pages = {315--329},
 year = {1988},
 language = {French},
 doi = {10.1007/BF01217743},
 zbMATH = {4150923},
 Zbl = {0702.35186}
}

@article{HelfNic,
 author = {Helffer, Bernard and Nicoleau, Fran{\c{c}}ois},
 title = {Trace formulas for the magnetic {Laplacian} and {Dirichlet} to {Neumann} operator -- explicit expansions},
 fjournal = {Asymptotic Analysis},
 journal = {Asymptotic Anal.},
 issn = {0921-7134},
 volume = {141},
 number = {2},
 pages = {95--118},
 year = {2025},
 language = {English},
 doi = {10.1177/09217134241308420},
 zbMATH = {8022911}
}

@article{Hopf,
 author = {Hopf, Eberhard},
 title = {Closed surfaces without conjugate points},
 fjournal = {Proceedings of the National Academy of Sciences of the United States of America},
 journal = {Proc. Natl. Acad. Sci. USA},
 issn = {0027-8424},
 volume = {34},
 pages = {47--51},
 year = {1948},
 language = {English},
 doi = {10.1073/pnas.34.2.47},
 url = {europepmc.org/articles/pmc1062913},
 zbMATH = {3046752},
 Zbl = {0030.07901}
}

@Book{Horm1,
 Author = {H{\"o}rmander, Lars},
 Title = {The analysis of linear partial differential operators. {III}: {Pseudo}-differential operators},
 FSeries = {Grundlehren der Mathematischen Wissenschaften},
 Series = {Grundlehren Math. Wiss.},
 Volume = {274},
 Year = {1985},
 Publisher = {Springer-Verlag},
 zbMATH = {3969124},
 Zbl = {0601.35001},
 DOI = {10.1007/978-3-540-49938-1}
}

@book{Horm2,
 Author = {H{\"o}rmander, Lars},
 Title = {The analysis of linear partial differential operators. {IV}: {Fourier} integral operators},
 FSeries = {Grundlehren der Mathematischen Wissenschaften},
 Series = {Grundlehren Math. Wiss.},
 ISSN = {0072-7830},
 Volume = {275},
 Year = {1985},
 Publisher = {Springer, Cham},
 Language = {English},
 zbMATH = {3989863},
 Zbl = {0612.35001},
 DOI = {10.1007/978-3-642-00136-9}
}

@Book{KH,
 Author = {Katok, Anatole and Hasselblatt, Boris},
 Title = {Introduction to the modern theory of dynamical systems. {With} a supplement by {Anatole} {Katok} and {Leonardo} {Mendoza}},
 FSeries = {Encyclopedia of Mathematics and Its Applications},
 Series = {Encycl. Math. Appl.},
 ISSN = {0953-4806},
 Volume = {54},
 ISBN = {0-521-57557-5},
 Year = {1997},
 Publisher = {Cambridge: Cambridge University Press},
 Language = {English},
 zbMATH = {989291},
 Zbl = {0878.58019}
}

@Misc{Kli,
 Author = {Klingenberg, Wilhelm},
 Title = {Lectures on closed geodesics. ({Lektsii} o zamknutykh geodezicheskikh). {Transl}. from the {Engl}},
 Year = {1982},
 Language = {Russian},
 HowPublished = {Moskva: {Izdatel}'stvo ''{Mir}''. 414 p. {R}. 3.40 (1982).},
 zbMATH = {3818697},
 Zbl = {0517.58004}
}

@Article{Kl,
 Author = {Klingenberg, Wilhelm},
 Title = {Riemannian manifolds with geodesic flow of {Anosov} type},
 FJournal = {Annals of Mathematics. Second Series},
 Journal = {Ann. Math. (2)},
 ISSN = {0003-486X},
 Volume = {99},
 Pages = {1--13},
 Year = {1974},
 Language = {English},
 DOI = {10.2307/1971011},
 zbMATH = {3428163},
 Zbl = {0272.53025}
}

@InCollection{Kni,
 Author = {Knieper, Gerhard},
 Title = {Hyperbolic dynamics and {Riemannian} geometry},
 BookTitle = {Handbook of dynamical systems. Volume 1A},
 ISBN = {0-444-82669-6},
 Pages = {453--545},
 Year = {2002},
 Publisher = {Amsterdam: North-Holland},
 Language = {English},
 zbMATH = {2079134},
 Zbl = {1049.37020}
}

@article{Kuw,
 author = {Kuwabara, Ruishi},
 title = {Isospectral connections on line bundles},
 fjournal = {Mathematische Zeitschrift},
 journal = {Math. Z.},
 issn = {0025-5874},
 volume = {204},
 number = {4},
 pages = {465--473},
 year = {1990},
 language = {English},
 doi = {10.1007/BF02570886},
 url = {https://eudml.org/doc/183783},
 zbMATH = {4203196},
 Zbl = {0728.53025}
}

@Article{LU,
 Author = {Lee, John M. and Uhlmann, Gunther},
 Title = {Determining anisotropic real-analytic conductivities by boundary measurements},
 FJournal = {Communications on Pure and Applied Mathematics},
 Journal = {Commun. Pure Appl. Math.},
 Volume = {42},
 Number = {8},
 Pages = {1097--1112},
 Year = {1989},
 DOI = {10.1002/cpa.3160420804},
 zbMATH = {4150775},
 Zbl = {0702.35036}
}

@article{LiuTan,
 author = {Liu, Genqian and Tan, Xiaoming},
 title = {Spectral invariants of the magnetic {Dirichlet}-to-{Neumann} map on {Riemannian} manifolds},
 fjournal = {Journal of Mathematical Physics},
 journal = {J. Math. Phys.},
 issn = {0022-2488},
 volume = {64},
 number = {4},
 pages = {47},
 note = {Id/No 041501},
 year = {2023},
 language = {English},
 doi = {10.1063/5.0088549},
 zbMATH = {7681832},
 Zbl = {1512.35419}
}

@article{NSU,
 author = {Nakamura, Gen and Sun, Ziqi and Uhlmann, Gunther},
 title = {Global identifiability for an inverse problem for the {Schr{\"o}dinger} equation in a magnetic field},
 fjournal = {Mathematische Annalen},
 journal = {Math. Ann.},
 issn = {0025-5831},
 volume = {303},
 number = {3},
 pages = {377--388},
 year = {1995},
 language = {English},
 doi = {10.1007/BF01460996},
 url = {https://eudml.org/doc/165373},
 zbMATH = {840191},
 Zbl = {0843.35134}
}

@Article{Pat,
 Author = {Paternain, Gabriel Pedro},
 Title = {Transparent connections over negatively curved surfaces},
 FJournal = {Journal of Modern Dynamics},
 Journal = {J. Mod. Dyn.},
 ISSN = {1930-5311},
 Volume = {3},
 Number = {2},
 Pages = {311--333},
 Year = {2009},
 Language = {English},
 DOI = {10.3934/jmd.2009.3.311},
 zbMATH = {5679588},
 Zbl = {1185.53024}
}

@Book{PatSalUhlbook,
 Author = {Paternain, Gabriel Pedro and Salo, Mikko and Uhlmann, Gunther},
 Title = {Geometric inverse problems. {With} emphasis on two dimensions},
 FSeries = {Cambridge Studies in Advanced Mathematics},
 Series = {Camb. Stud. Adv. Math.},
 Volume = {204},
 ISBN = {978-1-316-51087-2; 978-1-00-903990-1},
 Year = {2023},
 Publisher = {Cambridge: Cambridge University Press},
 Language = {English},
 DOI = {10.1017/9781009039901},
 zbMATH = {7625517},
 Zbl = {1519.35005}
}

@Article{PatSalUhl2,
 Author = {Paternain, Gabriel Pedro and Salo, Mikko and Uhlmann, Gunther},
 Title = {Spectral rigidity and invariant distributions on {Anosov} surfaces},
 FJournal = {Journal of Differential Geometry},
 Journal = {J. Differ. Geom.},
 ISSN = {0022-040X},
 Volume = {98},
 Number = {1},
 Pages = {147--181},
 Year = {2014},
 Language = {English},
 DOI = {10.4310/jdg/1406137697},
 zbMATH = {6322530},
 Zbl = {1304.37021}
}

@article{ProSav,
 author = {Provenzano, Luigi and Savo, Alessandro},
 title = {Geometry of the magnetic {Steklov} problem on {Riemannian} annuli},
 fjournal = {Communications in Contemporary Mathematics},
 journal = {Commun. Contemp. Math.},
 issn = {0219-1997},
 volume = {27},
 number = {7},
 pages = {49},
 note = {Id/No 2550001},
 year = {2025},
 language = {English},
 doi = {10.1142/S0219199725500014},
 zbMATH = {8052238}
}

@misc{SalSurvey,
 author = {Salo, Mikko},
 title = {Inverse boundary value problems for the magnetic {Schroedinger} equation},
 year = {2006},
 howpublished = {Preprint, {arXiv}:math/0611458 [math.{AP}] (2006)},
 url = {https://arxiv.org/abs/math/0611458},
 arXiv = {arXiv:math/0611458}
}

@article{Sasaki,
 author = {Sasaki, Shigeo},
 title = {On the differential geometry of tangent bundles of {Riemannian} manifolds},
 fjournal = {T{\^o}hoku Mathematical Journal. Second Series},
 journal = {T{\^o}hoku Math. J. (2)},
 issn = {0040-8735},
 volume = {10},
 pages = {338--354},
 year = {1958},
 language = {English},
 doi = {10.2748/tmj/1178244668},
 zbMATH = {3140971},
 Zbl = {0086.15003}
}

@article{Se,
 author = {Serov, V. S.},
 title = {Borg-Levinson theorem for magnetic {Schr{\"o}dinger} operator},
 fjournal = {Bulletin of the Greek Mathematical Society},
 journal = {Bull. Greek Math. Soc.},
 issn = {0072-7466},
 volume = {57},
 pages = {321--332},
 year = {2010},
 language = {English},
 zbMATH = {6923841},
 Zbl = {1393.35022}
}

@Book{Sh,
 Author = {Sharafutdinov, Vladimir A.},
 Title = {Integral geometry for tensor fields. {Transl}. from the {Russian}},
 FSeries = {Inverse and Ill-Posed Problems Series},
 Series = {Inverse Ill-Posed Probl. Ser.},
 ISSN = {1381-4524},
 ISBN = {90-6764-165-0},
 Year = {1994},
 Publisher = {Utrecht: VSP},
 Language = {English},
 zbMATH = {844036},
 Zbl = {0883.53004}
}

@article{S,
 author = {Sharafutdinov, V. A.},
 title = {Some problems of integral geometry on {Anosov} manifolds.},
 fjournal = {Doklady Mathematics},
 journal = {Dokl. Math.},
 issn = {1064-5624},
 volume = {60},
 number = {2},
 pages = {18--20},
 year = {1999},
 language = {English},
 zbMATH = {2060240},
 Zbl = {1072.53542}
}

@article{Sha,
 author = {Sharafutdinov, V. A.},
 title = {Local audibility of a hyperbolic metric},
 fjournal = {Sibirski{\u{\i}} Matematicheski{\u{\i}} Zhurnal},
 journal = {Sib. Mat. Zh.},
 issn = {0037-4474},
 volume = {50},
 number = {5},
 pages = {1176--1194},
 year = {2009},
 language = {Russian; English},
 url = {https://eudml.org/doc/232210},
 zbMATH = {5875575},
 Zbl = {1224.58023}
}

@article{Shi,
 author = {Shigekawa, Ichiro},
 title = {Eigenvalue problems for the {Schr{\"o}dinger} operator with the magnetic field on a compact {Riemann} manifold},
 fjournal = {Journal of Functional Analysis},
 journal = {J. Funct. Anal.},
 issn = {0022-1236},
 volume = {75},
 pages = {92--127},
 year = {1987},
 language = {English},
 doi = {10.1016/0022-1236(87)90108-X},
 zbMATH = {4024353},
 Zbl = {0629.58023}
}

@Article{Sun,
 Author = {Sunada, Toshikazu},
 Title = {Riemannian coverings and isospectral manifolds},
 FJournal = {Annals of Mathematics. Second Series},
 Journal = {Ann. Math. (2)},
 ISSN = {0003-486X},
 Volume = {121},
 Pages = {169--186},
 Year = {1985},
 Language = {English},
 DOI = {10.2307/1971195},
 zbMATH = {3938113},
 Zbl = {0585.58047}
}

@article{T,
 author = {Tondeur, P.},
 title = {Structure presque k{\"a}hlerienne naturelle sur le fibre des vecteurs covariants d'une vari{\'e}t{\'e} riemannienne},
 fjournal = {Comptes Rendus Hebdomadaires des S{\'e}ances de l'Acad{\'e}mie des Sciences, Paris},
 journal = {C. R. Acad. Sci., Paris},
 issn = {0001-4036},
 volume = {254},
 pages = {407--408},
 year = {1962},
 language = {French},
 zbMATH = {3182575},
 Zbl = {0112.36901}
}

@article{V,
 author = {Vign{\'e}ras, Marie-France},
 title = {Isospectral and non-isometric {Riemannian} manifolds},
 fjournal = {Annals of Mathematics. Second Series},
 journal = {Ann. Math. (2)},
 issn = {0003-486X},
 volume = {112},
 pages = {21--32},
 year = {1980},
 language = {French},
 doi = {10.2307/1971319},
 zbMATH = {3696033},
 Zbl = {0445.53026}
}

@InCollection{Zeld,
 Author = {Zelditch, Steve},
 Title = {The inverse spectral problem},
 BookTitle = {Surveys in differential geometry. Vol. IX: Eigenvalues of Laplacians and other geometric operators},
 ISBN = {1-57146-115-9},
 Pages = {401--467},
 Year = {2004},
 Publisher = {Somerville, MA: International Press},
 Language = {English},
 zbMATH = {2133161},
 Zbl = {1061.58029}
}

@misc{Zeld15,
 author = {Zelditch, Steve},
 title = {Survey of the inverse spectral problem for the {Notices} of the {ICCM}},
 year = {2015},
 howpublished = {Preprint, {arXiv}:1504.02000 [math.{SP}] (2015)},
 url = {https://arxiv.org/abs/1504.02000},
 arXiv = {arXiv:1504.02000}
}

\end{document}